\documentclass[preprint, aos]{imsart}

\usepackage{amsmath, amsfonts, amssymb, amsthm, amsxtra}
\usepackage{mathrsfs, mathtools}
\usepackage{bbm}
\usepackage{graphicx}
\usepackage{lmodern}

\RequirePackage[colorlinks,citecolor=blue,urlcolor=blue]{hyperref}

\startlocaldefs

\numberwithin{equation}{section}

\newtheorem{definition}{Definition}[section]
\newtheorem{theorem}[definition]{Theorem}
\newtheorem{lemma}[definition]{Lemma}
\newtheorem{corollary}[definition]{Corollary}
\newtheorem{remark}[definition]{Remark}
\newtheorem{proposition}[definition]{Proposition}

\newcommand{\cE}{\mathcal{E}}

\newcommand{\cS}{\mathcal{S}}

\newcommand{\cU}{\mathcal{U}}

\newcommand{\subG}{\operatorname{subG}}
\newcommand{\card}{\operatorname{\mathsf{card}}}
\newcommand{\R}{{\rm I}\kern-0.18em{\rm R}}
\newcommand{\h}{{\rm I}\kern-0.18em{\rm H}}
\newcommand{\K}{{\rm I}\kern-0.18em{\rm K}}
\newcommand{\p}{{\rm I}\kern-0.18em{\rm P}}
\newcommand{\E}{{\rm I}\kern-0.18em{\rm E}}
\newcommand{\Z}{{\rm Z}\kern-0.18em{\rm Z}}
\newcommand{\1}{{\rm 1}\kern-0.24em{\rm I}}
\newcommand{\N}{{\rm I}\kern-0.18em{\rm N}}
\newcommand{\cop}{\textsf{C1P}}

\newcommand{\RankSum}{\textsf{RankSum}}
\newcommand{\RankScore}{\textsf{RankScore}}

\endlocaldefs

\begin{document}

\begin{frontmatter}

\title{Optimal rates of Statistical Seriation}
\runtitle{Statistical Seriation}

\begin{aug}

\author{\fnms{Nicolas}~\snm{Flammarion}\ead[label=flammarion]{nicolas.flammarion@ens.fr}},
\author{\fnms{Cheng}~\snm{Mao}\ead[label=mao]{maocheng@math.mit.edu}}
\and
\author{\fnms{Philippe}~\snm{Rigollet}\ead[label=rigollet]{rigollet@math.mit.edu}}

\affiliation{Ecole Normale Sup\'erieure and Massachusetts Institute of Technology}

%
%

\runauthor{Flammarion, Mao and Rigollet}
\end{aug}

\begin{abstract}
Given a matrix, the seriation problem consists in permuting its rows in such way that all its columns have the same shape, for example, they are monotone increasing. We propose a statistical approach to this problem where the matrix of interest is observed with noise and study the corresponding minimax rate of estimation of the matrices. Specifically, when the columns are either unimodal or monotone, we show that the least squares estimator is optimal up to logarithmic factors and adapts to matrices with a certain natural structure. Finally, we propose a computationally efficient estimator in the monotonic case and study its performance both theoretically and experimentally. Our work is at the intersection of shape constrained estimation and recent work that involves permutation learning, such as graph denoising and ranking.
\end{abstract}



\end{frontmatter}

\section{Introduction}  \label{sec:intro}

The \emph{consecutive 1's problem} (\cop) \cite{FulGro64} is defined as follows. Given a binary matrix $A$ the goal is to permute its rows in such a way that the resulting matrix enjoys the \emph{consecutive 1's property}: each of its columns is a vector $v=(v_1, \ldots, v_n)^\top$ where $v_j=1$ if and only if $a\le j\le b$ for two integers $a,b$ between $1$ and $n$. 

This problem has its roots in archeology and especially \emph{sequence dating} where the goal is to recover the chronological order of sepultures based on artifacts found in these sepultures where the entry $A_{i,j}$ of matrix $A$ indicates the presence of artifact $j$ in sepulture $i$. In his seminal work, egyptologist Flinders Petrie~\cite{Pet99} formulated the hypothesis that two sepultures should be close in the time domain if they present similar sets of artifacts. Already in the noiseless case, this problem presents an interesting algorithmic challenge and is reducible to the famous Travelling Salesman Problem~\cite{GerGro12} as observed by statistician David Kendall \cite{Ken63,Ken69,Ken70,Ken71}  who employed early tools  from multidimensional scaling as a heuristic to solve it.  \cop\ belongs to a more general class of so-called \emph{seriation} problems that consist in optimizing various criteria over the discrete set of permutations. While such problems are hard in general, it can be shown that a 
subset of the these problems, including \cop, can be solve efficiently using spectral method~\cite{AtkBomHen98} or convex optimization~\cite{FogJenBac13, LimWri14}. However, little is known about the robustness to noise of such methods. 

%

In order to set the benchmark for the noisy case, we propose a statistical \emph{seriation model} and study optimal rates of estimation in this model. Assume that we observe an $n \times m$ matrix $Y=\Pi A + Z$,
where $\Pi$ is an unknown $n \times n$ permutation matrix, $Z$ is an $n \times m$ noise matrix and $A \in {\R}^{n \times m}$ is assumed to belong to a class of matrices that satisfy a certain shape constraint. 
Our goal is to give estimators $\hat \Pi$ and $\hat A$ so that $\hat \Pi \hat A$ is close to $\Pi A$.
The shape constraint can be the consecutive 1's property, but more generally, we consider the class of matrices that have unimodal columns, which also include monotonic columns as a special case. These terms will be formally defined at the end of this section.

The rest of the paper is organized as follows. In Section~\ref{sec:problem} we formulate the model and discuss related work. Section~\ref{sec:result} collects our main results, including uniform and adaptive upper bounds for the least squares estimator together with corresponding minimax lower bounds in the general unimodal case. In Section~\ref{sec:rankscore}, for the special case of monotone columns, we propose a computationally efficient alternative to the least squares estimator and study its rates of convergence both theoretically and numerically. Appendix~\ref{sec:upper} is devoted to the proofs of the upper bounds, which use the metric entropy bounds proved in Appendix~\ref{sec:covering}. The proofs of the information-theoretic lower bounds are presented in Appendix~\ref{sec:lower}. In Appendix~\ref{sec:monotone}, we study the rate of estimation of the efficient estimator for the monotonic case. Appendix~\ref{sec:trivial-upper} contains a delayed proof of a trivial upper bound. Appendix~\ref{sec:unimodal} presents new bounds for unimodal regression implied by our analysis, which are minimax optimal up to logarithmic factors. 
%

\medskip


%

 \textsc{Notation.}
For a positive integer $n$, define $[n]=\{1,\ldots, n\}$. For a matrix $A \in {\R}^{n \times m}$, 
let $\|A\|_F$ denote its Frobenius norm, and 
let $A_{i,\cdot}$ be its $i$-th row and $A_{\cdot,j}$ be its $j$-th column.
Let $\mathcal B^n(a, t)$ denote the Euclidean ball of radius $t$ centered at $a$ in ${\R}^n$. We use $C$ and $c$ to denote positive constants that may change from line to line. For any two sequences $(u_n)_n$ and $(v_n)_n$, we write $u_n \lesssim v_n$ if there exists an absolute constant $C>0$ such that $u_n \le C v_n$ for all $n$. We define $u_n \gtrsim v_n$ analogously. Given two real numbers $a, b$, define $a\wedge b=\min(a,b)$ and $a\vee b=\max(a,b)$.

Denote the closed convex cone of  increasing\footnote{Throughout the paper, we loosely use the terms ``increasing" and ``decreasing" to mean ``monotonically non-decreasing" and ``monotonically non-increasing" respectively.} sequences in ${\R}^n$ by $\mathcal S_n = \{a \in {\R}^n: a_1 \le \cdots \le a_n\}$.
We define $\mathcal S^m$ to be the Cartesian product of $m$ copies of $\mathcal S_n$ and we identify $\mathcal S^m$ to the set of $n \times m$ matrices with increasing columns.

For any $l\in[n]$, define the closed convex cone $\mathcal C_l=\{a \in {\R}^n: a_1 \le \cdots \le a_l\}\cap\{a \in {\R}^n: a_l\geq \cdots \geq a_n\}$, which consists of vectors in ${\R}^n$ that increase up to the $l$-th entry and then decrease. 
Define the set $\mathcal U$ of unimodal sequences in ${\R}^n$ by $\mathcal U=\bigcup _{l=1}^n \mathcal C_l$.  We define $\mathcal U^m$ to be the Cartesian product of $m$ copies of $\mathcal U$ and we identify $\mathcal U^m$ to the set of $n \times m$ matrices with unimodal columns. It is also convenient to write $\mathcal U^m$ as a union of closed convex cones as follows.
For $\mathbf l = (l_1, \dots, l_m) \in [n]^m$,  let $\mathcal C^m_\mathbf{l}=\mathcal C_{l_1}\times \cdots\times \mathcal C_{l_m}$. Then $\mathcal U^m$ is the union  of the $n^m$ closed convex cones $\mathcal C^m_\mathbf{l}, \mathbf l  \in [n]^m$.

Finally, let $\mathfrak S_n$ be the set of $n \times n$ permutation matrices and define $\mathcal M = \bigcup_{\Pi \in \mathfrak S_n} \Pi \mathcal U^m$  where $\Pi \mathcal U^m = \{ \Pi A: A \in \mathcal U^m\}$, so that $\mathcal{M}$ is the union of the $n!n^m$ closed convex cones  $\Pi\mathcal C^m_\mathbf{l}, \Pi \in \mathfrak S_n, \mathbf l  \in [n]^m$.

\section{Problem setup and related work} \label{sec:problem} 

In this section, we formally state the problem of interest and discuss several lines of related work.

\subsection{The seriation model} \label{sec:setup}

Suppose that we observe a matrix $Y \in {\R}^{n\times m}$, $n \ge 2$ such that
\begin{equation} \label{eq:model}
Y = \Pi^* A^* + Z\, , 
\end{equation} 
where $A^* \in \mathcal U^m$, $\Pi \in \mathfrak S_n$ and $Z$ is a centered sub-Gaussian noise matrix with variance proxy $\sigma^2>0$. More specifically, $Z$ is a matrix such that ${\E}[Z]=0$ and, for any $M\in {\R}^{n\times m}$,
\[
{\E}\big[\exp\big(\mathsf{Tr}(Z^\top M)\big)\big] \le \exp\Big(\frac{\sigma^2\|M\|_F^2}{2}\Big)\, ,
\]
where $\mathsf{Tr}(\cdot)$ is the trace operator.
We write  $Z\sim \operatorname{subG}_{n,m}(\sigma^2)$ or simply $Z\sim \operatorname{subG}(\sigma^2)$ when dimensions are clear from the context.

Given the observation $Y$, our goal is to estimate the unknown pair $(\Pi^*, A^*)$. The performance of an estimator $(\hat \Pi, \hat A) \in \mathfrak S_n \times \mathcal U^m$, is measured by the quadratic loss:
\[\frac{1}{nm} \| \hat \Pi\hat A - \Pi^* A^*\|_F^2\, . \]
In particular, its expectation is the mean squared error.
Since we are interested in estimating $\Pi^* A^* \in \mathcal M$, we can also view $\mathcal M$ as the parameter space.

In the general unimodal case,  upper bounds on the above quadratic loss do not imply individual upper bounds on estimation of the matrix $\Pi^*$ or the matrix $A^*$ due to lack of identifiability. Nevertheless, if we further assume that the columns of $A^*$ are monotone increasing, that is $A^* \in \cS^m$, then the following lemma holds.
\begin{lemma} \label{lem:rearrange}
If $A^*, \tilde A \in \mathcal S^m$, then for any $\Pi^*, \tilde \Pi \in \mathfrak S_n$, we have that
\[
\|\tilde A - A^*\|_F^2 \le \|\tilde \Pi \tilde A - \Pi^* A^*\|_F^2 \,, 
\]
and that
\[
\|\tilde \Pi A^* - \Pi^* A^*\|_F^2 \le 4 \|\tilde \Pi \tilde A - \Pi^* A^*\|_F^2 \,.
\]
\end{lemma}

\begin{proof}
Let $a, b \in \mathcal S_n$ and $b_{\pi} = (b_{\pi(1)}, \dots, b_{\pi(n)})$ where $\pi:[n]\to [n]$ is a permutation. It is easy to check that $\sum_{i=1}^n a_i b_i \ge \sum_{i=1}^n a_i b_{\pi(i)}$, so $\|a - b\|_2^2 \le \|a - b_\pi \|_2^2$. 
Applying this inequality to columns of matrices, we see that
\[ \|\tilde A - A^*\|_F^2 \le \|\tilde A - \tilde \Pi^{-1} \Pi^* A^*\|_F^2 = \|\tilde \Pi \tilde A - \Pi^* A^*\|_F^2, \]
since $A^*, \tilde A \in \mathcal S^m$.
Moreover, $\|\tilde \Pi A^* - \tilde \Pi \tilde A\|_F = \|A^* - \tilde A\|_F$, so
\[ \|\tilde \Pi A^* - \Pi^* A^*\|_F \le \|A^* - \tilde A\|_F + \|\tilde \Pi \tilde A - \Pi^* A^*\|_F \le 2 \|\tilde \Pi \tilde A - \Pi^* A^*\|_F,
\]
by the triangle inequality and the previous display.
\end{proof}

Lemma~\ref{lem:rearrange} guarantees that $\|\tilde \Pi A^* - \Pi^* A^*\|_F$ is a pertinent measure of the performance of $\tilde \Pi$. Note further that $\|\tilde \Pi A^* - \Pi^* A^*\|_F$ is large if $\tilde \Pi$ misplaces rows of $A^*$ that have large differences, and is small if $\tilde \Pi$ only misplaces rows of $A^*$ that are close to each other. We argue that, in the seriation context, this measure of distance between permutations is more natural than ad hoc choices such as the trivial 0/1 distance or popular choices such as Kendall's $\tau$ or Spearman's $\rho$.

Apart from Section~\ref{sec:rankscore} (and Appendix~\ref{sec:monotone}), the rest of this paper focuses on the least squares (LS) estimator defined by
\begin{equation} \label{def-pi-a} 
(\hat \Pi, \hat A) \in \operatorname*{argmin}_{(\Pi, A) \in {\mathfrak S_n}\times \mathcal U^m} \|Y - \Pi A\|_F^2\, .
\end{equation}
Taking $\hat M = \hat \Pi \hat A$, we see that it is equivalent to define the LS estimator by
\begin{equation} \label{def-pia}
\hat M \in \operatorname*{argmin}_{M \in \mathcal M} \|Y - M\|_F^2 \, .
\end{equation}
Note that in our case, the set of parameters $\mathcal M$ is not convex, but is a union of $n!n^m$ closed convex cones and it is not clear how to compute the LS estimator efficiently. We discuss this aspect in further details in the context of monotone columns in Section~\ref{sec:rankscore}. Nevertheless, the main focus of this paper is the least squares estimator which, as we shall see, is near-optimal in a minimax sense and therefore serves as a benchmark for the statistical seriation model.

\subsection{Related work} \label{sec:related-work}

Our work falls broadly in the scope of statistical inference under shape constraints but presents a major twist: the unknown latent permutation $\Pi^*$. 

\subsubsection{Shape constrained regression} To set our goals, we first consider the case where the permutation is known and assume without loss of generality that $\Pi^*=I_n$.  In this case, we can estimate individually each column $A^*_{\cdot,j}$ by an estimator $\hat A_{\cdot,j}$ and then get an estimator $\hat A$ for the whole matrix  by concatenating the columns $\hat A_{\cdot,j}$.  Thus the task is reduced to estimation of  a vector $\theta^*$ which satisfies a certain shape constraint from an observation $y=\theta^*+z$ where $z\sim \subG_{n,1}(\sigma^2)$. 

When  $\theta^*$ is assumed to be increasing we speak of isotonic regression   \cite{BarBarBre72}.  The LS estimator defined by $\hat \theta =\operatorname*{argmin}_{\theta \in \mathcal S_n} \| \theta -y\|_2^2 $ can be computed in closed form in  $O(n)$ using the Pool-Adjacent-Violators algorithm (PAVA) \cite{AyeBruEwi55, BarBarBre72,RobWriDyk88}  and its statistical performance  has been studied  by Zhang \cite{Zha02} (see also \cite{NemPolTsy85, Don90, Gee90,  Mam91,  Gee93} for similar bounds using empirical process theory) who   showed in the Gaussian case $z \sim N(0, \sigma^2I_n)$ that the mean squared error behaves like
\begin{equation} \label{eq:isotonic-global}
\frac{1}{n}{\E}\|\hat \theta-\theta^*\|_2^2 \asymp \Big(\frac{\sigma^2 V(\theta^*)}{n}\Big)^{2/3} \,,
\end{equation}
where $V(\theta)=\max_{i \in[n]} \theta_i - \min_{i \in [n]} \theta_i$ is the variation of $\theta \in {\R}^n$.
Note that $2/3=2\beta/(2\beta+1)$ for $\beta=1$ so that this is the minimax rate of estimation of Lipschitz functions (see, e.g., \cite{Tsy09}). 

The rate in \eqref{eq:isotonic-global} is said to be \emph{global} has it holds uniformly over the set of monotone vectors with variation $V(\theta^*)$. Recently, \cite{ChaGunSen15} have initiated the study of \emph{adaptive} bounds that may be better if $\theta^*$ has a simpler structure in some sense. To define this structure, let $k(\theta)= \card(\{\theta_1,\cdots,\theta_n\})$ denote the cardinality of entries of $\theta \in {\R}^n$. In this context, \cite{ChaGunSen15} showed that the LS estimator satisfies the adaptive bound
\begin{equation} \label{eq:isotonic}
\frac{1}{n}{\E}\|\hat \theta-\theta^*\|_2^2\leq   C\inf_{\theta\in \mathcal S_n }\Big(\frac{\|\theta-\theta^*\|^2}{n}+\frac{\sigma^2 k(\theta)}{n}\log \frac{en}{k(\theta)}\Big) \,.
\end{equation}
This result was extended in \cite{Bel15}  to a sharp oracle inequality where $C=1$. This bound was also shown to be optimal in a minimax sense \cite{ChaGunSen15, BelTsy15}.

Unlike its monotone counterpart, unimodal regression  where $\theta^* \in \mathcal U$ has received  sporadic attention \cite{ShoZha01, Kol14, ChaLaf15}. This state of affairs is all the more surprising given that unimodal density estimation has been the subject of much more research \cite{BicFan96, Bir97, EggLaR00, DasDiaSer12,DasDiaSer13,TurGho14}. It was recently shown in \cite{ChaLaf15} that the LS estimator also adapts to $V(\theta^*)$ and $k(\theta^*)$ for unimodal regression:
\begin{equation} \label{eq:unimodal}
\frac{1}{n}\|\hat \theta-\theta^*\|_2^2\lesssim \min\Big( \sigma^{4/3}\Big( \frac{V(\theta^*)+\sigma}{n}\Big)^{2/3}, \frac{\sigma^2}{n}k(\theta^*)^{3/2}(\log n) ^{3/2}\Big) 
\end{equation}
with probability at least $1-n^{-\alpha}$ for some $\alpha > 0$. The exponent $3/2$ in the second term was improved to $1$ in the new version of \cite{ChaLaf15} after the first version of our current paper was posted. Note that the exponents in \eqref{eq:unimodal} are different from the isotonic case. Our results will imply that they are not optimal and in fact the LS estimator achieves the same rate as in isotonic regression. See Corollary~\ref{cor:unimodal} for more details. 
The algorithmic aspect of unimodal regression has received more attention  \cite{Fri80,GenShi90,BroSid98,BoyMag06} and  \cite{Sto08} showed that the LS estimator can be computed with time complexity $O(n)$ using a modified version of PAVA. 
Hence there is little difference between isotonic and unimodal regressions from both  computational and statistical points of views.

\subsubsection{Latent permutation learning} When  the permutation $\Pi^*$ is unknown the estimation problem is more involved. Noisy permutation learning was explicitly addressed in \cite{ColDal16} where the problem of matching two sets of noisy vectors was studied from a statistical point of view. Given $n \times  m$ matrices  $Y=A^*+Z$ and $\tilde Y = \Pi^* A^* + \tilde Z$, where $A^* \in{\R}^{n\times m}$ is an unknown matrix and $\Pi^*  \in{\R}^{n\times n}$ is an unknown permutation matrix, the goal is to recover $\Pi^*$. It was shown in \cite{ColDal16} that if   $\min_{i\neq j} \|A_{i, \cdot} -A_{j, \cdot}\|_2 \ge  c \sigma \big( (\log n)^{1/2} \vee (m \log n)^{1/4} \big)$, then the LS estimator defined by $\hat \Pi=\operatorname*{argmin}_{\Pi\in\mathfrak S_n} \|\Pi Y-\tilde Y\|_F^2$ recovers the true permutation with high probability. However they did not directly study the behavior of  $\| \hat \Pi A^* -\Pi^* A^* \|^2_F$.

In his celebrated paper on matrix estimation~\cite{Cha15}, Sourav Chatterjee describes several noisy matrix models involving unknown latent permutations. One is the \emph{nonparametric Bradley-Terry-Luce} (NP-BTL) model where we observe a matrix $Y \in {\R}^{n \times n}$ with independent entries $Y_{i,j}\sim \operatorname{Ber}(P_{i,j})$ for some unknown parameters $P=\{P_{i,j}\}_{1\le i,j \le n}$ where $P_{i,j} \in [0,1]$ is equal to the probability that item $i$ is preferred over item $j$ and $P_{j,i}=1-P_{i,j}$. Crucially, the NP-BTL model assumes the so-called \emph{strong stochastic transitivity (SST)} \cite{DavMar59,Fis73} assumption: there exists an unknown permutation matrix $\Pi \in \R^{n \times n}$ such that the ordered matrix $A=\Pi^\top P \Pi$ satisfies $A_{1,k} \le \cdots \le A_{n,k}$ for all $k \in [n]$. Note that the NP-BTL model is a special case of our model \eqref{eq:model} where $m=n$ and $Z \sim \operatorname{subG}(1/4)$ is taken to be Bernoulli. Chatterjee proposed an estimator  $\hat P$ that leverages the 
fact that any matrix $P$ in the NP-BTL model 
can be approximated by a low rank matrix and proved \cite[Theorem~2.11]{Cha15} that $n^{-2} \|\hat P-P\|_F^2\lesssim n^{-1/4}$, which was improved to $n^{-1/2}$ by \cite{ShaBalGunWai15} for a variation of the estimator. This method does not yield individual estimators of $\Pi$ or $A$, and \cite{ChaMuk16} proposed estimators $\hat \Pi$ and $\hat A$  so that $\hat \Pi \hat A \hat\Pi^\top$ estimates $P$ with the same rate $n^{-1/2}$ up to a logarithmic factor. 
The non-optimality of this rate has been observed in \cite{ShaBalGunWai15} who showed that the correct rate should be of order $n^{-1}$ up to a possible $\log n$ factor. However, it is not known whether a computationally efficient estimator could achieve the fast rate. 
A recent work \cite{ShaBalWai16} explored a new notion of adaptivity for which the authors proved a computational lower bound, and also proposed an efficient estimator whose rate of estimation matches that lower bound.

Also mentioned in Chatterjee's paper is the so-called \emph{stochastic block model} that has since received such extensive attention in various communities that it is futile to attempt to establish a comprehensive list of references. Instead, we refer the reader to~\cite{GaoLuZho15} and references therein.  This paper establishes the minimax rates for this problem and its continuous limit, the graphon estimation problem and, as such, constitutes the state-of-the-art in the statistical literature. In the stochastic block model with $k \ge 2$ blocks, we assume that we observe a matrix $Y=P+Z$ where $P=\Pi A \Pi^\top, \Pi \in \R^{\times n}$ is an unknown permutation matrix and $A$ has a block structure, namely, there exist positive integers $n_1< \ldots <n_{k}<n_{k+1}:=n$, and $k^2$ real numbers $a_{s,t}, (s,t) \in [k]^2$ such that $A$ has entries 
$$
A_{i,j}=\sum_{(s,t) \in [k]^2}a_{s,t} \1\{n_s\le i \le n_{s+1}, n_t\le j \le n_{t+1}\}\,, \qquad i,j \in [n]\,.
$$
While traditionally, the stochastic block model is a network model and therefore pertains only to Bernoulli observations, the more general case of sub-Gaussian additive error is also explicitly handled in~\cite{GaoLuZho15}. For this problem, Gao, Liu and Zhou have established that the least squares estimator $\hat P$ satisfies $n^{-2} \|\hat P-P\|_F^2\lesssim k^2/n^2+(\log k)/n$ together with a matching lower bound. Using piecewise constant approximation to bivariate H\"older functions, they also establish that this estimator with a correct choice of $k$ leads to minimax optimal estimation of smooth graphons. 
Both results exploit extensively the fact that the matrix $P$ is equal to or can be well approximated by a piecewise constant matrix and our results below take a similar route by observing that monotone and unimodal vectors are also well approximated by piecewise constant ones. Moreover, we allow for rectangular matrices. 

In fact, our result can be also formulated as a network estimation problem but on a bipartite graph, thus falling at the intersection of the above two examples. Assume that $n$ left nodes represent items and that $m$ right nodes represent users. Assume further that we observe the $n \times m$ adjacency matrix $Y$ of a random graph where the presence of edge $(i,j)$ indicates that user $j$ has purchased or liked item $i$. Define $P=\E[Y]$ and assume SST across items in the sense that there exists an unknown $n \times n$ permutation matrix $\Pi^*$ such that $P=\Pi^* A^*$ and $A^*$ is such that $A^*_{1, j} \le \cdots \le A^*_{n, j}$ for all users $j \in [m]$. This model falls into the scope of the statistical seriation model \eqref{eq:model}.

\section{Main results} \label{sec:result}


\subsection{Adaptive oracle inequalities}

For a matrix $A \in \mathcal U^m$, let $k(A_{\cdot,j})=\card(\{A_{1,j}, \dots, A_{n,j}\})$ be the number of values taken by the $j$-th column of $A$ and define $K(A) = \sum_{j=1}^m k(A_{\cdot,j})$. Observe that $K(A)\ge m$. The first theorem shows that the LS estimator adapts to the complexity $K$.

\begin{theorem} \label{thm:adaptive}
For $A^* \in {\R}^{n \times m}$ and $Y = \Pi^* A^* + Z$, let $(\hat \Pi, \hat A)$ be the LS estimator defined in \eqref{def-pi-a}. Then the following oracle inequality holds
\begin{equation} \label{eq:oracle-loss} 
\frac{1}{nm} \| \hat \Pi\hat A - \Pi^* A^*\|_F^2 \lesssim   \min_{A \in \mathcal U^m} \Big(\frac 1{nm} \|A - A^*\|_F^2 + \sigma^2 \frac {K(A)}{nm} \log \frac{enm}{K(A)}\Big) + \sigma^2 \frac{\log n}m
\end{equation}
with probability at least $1-e^{-c(n+m)}, c>0$. Moreover, 
\begin{equation} \label{eq:oracle-mse} 
\frac{1}{nm} \E\| \hat \Pi\hat A - \Pi^* A^*\|_F^2 \lesssim   \min_{A \in \mathcal U^m} \Big(\frac 1{nm} \|A - A^*\|_F^2 + \sigma^2 \frac {K(A)}{nm} \log \frac{enm}{K(A)}\Big) + \sigma^2 \frac{\log n}m \, .
\end{equation}
\end{theorem}
Note that while we assume that $A^* \in \mathcal U^m$ in \eqref{eq:model}, the above oracle inequalities hold in fact for any $A^* \in {\R}^{n \times m}$ even if its columns are \emph{not} assumed to be unimodal. 

The above oracle inequalities indicate that the LS estimator automatically trades off the approximation error $\|A - A^*\|_F^2$ for the stochastic error $\sigma^2 K(A) \log (enm/K(A))$.

If $A^*$ is assumed to have unimodal columns, then we can take $A = A^*$ in \eqref{eq:oracle-loss} and \eqref{eq:oracle-mse} to get the following corollary.

\begin{corollary} \label{cor:adaptive}
For $A^* \in \mathcal U^m$ and $Y = \Pi^* A^* + Z$, the LS estimator $(\hat \Pi, \hat A)$ satisfies 
\[
\frac{1}{nm} \| \hat \Pi\hat A - \Pi^* A^*\|_F^2  \lesssim \sigma^2 \Big(\frac {K(A^*)}{nm} \log \frac{enm}{K(A^*)} + \frac{\log n}m \Big)
\]
with probability at least $1- e^{-c(n+m)}, c>0$. 
Moreover, the corresponding bound with the same rate holds in expectation.
\end{corollary}

The two terms in the adaptive bound can be understood as follows. The first term corresponds to the estimation of the matrix $A^*$ with unimodal columuns if the permutation $\Pi^*$ is known. It can be viewed as a matrix version of the adaptive bound \eqref{eq:isotonic} in the vector case. The LS estimator adapts to the cardinality of entries of $A^*$ as it achieves a provably better rate if $K(A^*)$ is smaller while not requiring knowledge of $K(A^*)$. The second term corresponds to the error due to the unknown permutation $\Pi^*$. As $m$ grows to infinity this second term vanishes, because we have more samples to estimate $\Pi^*$ better. If $m \ge n$, it is easy to check that the permutation term is dominated by the first term, so the rate of estimation is the same as if the permutation is known.

\subsection{Global oracle inequalities}

The bounds in Theorem~\ref{thm:adaptive} adapt to the cardinality of the oracle. 
In this subsection, we state another type of upper bounds for the LS estimator $(\hat \Pi, \hat A)$.
They are called global bounds because they hold uniformly over the class of matrices whose columns are unimodal and that have bounded  variation.
Recall that we call \emph{variation} of a vector $a \in {\R}^n$ the scalar $V(a) \ge 0$ defined by
\[
V(a) = \max_{1\le i \le n} a_i - \min_{1 \le i \le n} a_i \,.
\]
We extend this notion to a matrix $A \in {\R}^{n\times m}$ by defining
\[
V(A) =  \Big(\frac 1m \sum_{j=1}^m V(A_{\cdot,j})^{2/3}\Big)^{3/2}\, . 
\]
While this $2/3$-norm may seem odd at first sight, it turns out to be the correct extrapolation from vectors to matrices, at least in the context under consideration here. Indeed, the following upper bound, in which this quantity naturally appears, is matched by the lower bound of Theorem~\ref{thm:lower-global} up to logarithmic terms.

\begin{theorem} \label{thm:global}
For $A^* \in {\R}^{n \times m}$ and $Y = \Pi^* A^* + Z$, let $(\hat \Pi, \hat A)$ be the LS estimator defined in \eqref{def-pi-a}. Then it holds that
\begin{multline} \label{eq:global-loss}
\frac{1}{nm} \| \hat \Pi\hat A - \Pi^* A^*\|_F^2
\lesssim  \min_{A \in \mathcal U^m} \Big[\frac 1{nm} \|A - A^*\|_F^2 + \Big(\frac{\sigma^2 V(A)\log n}{n}\Big)^{2/3} \Big] + \sigma^2 \frac{\log n}{n\wedge m} \,.
\end{multline}
with probability at least $1-e^{-c(n+m)}, c>0$. Moreover, the corresponding bound with the same rate holds in expectation. 
\end{theorem}

If $A^* \in \mathcal U^m$, then taking $A = A^*$ in Theorem~\ref{thm:global} leads to the following corollary that indicates that the LS estimator is adaptive to the quantity~$V(A^*)$. 

\begin{corollary} \label{cor:global}
For $A^* \in \mathcal U^m$ and $Y = \Pi^* A^* + Z$, the LS estimator $(\hat \Pi, \hat A)$ satisfies 
\[
\frac{1}{nm} \| \hat \Pi\hat A - \Pi^* A^*\|_F^2\lesssim \Big(\frac{\sigma^2 V(A^*)\log n }{n}\Big)^{2/3}  + \sigma^2 \frac{\log n}{n\wedge m}
\]
with probability at least $1-e^{-c(n+m)}, c>0$. Moreover, the corresponding bound with the same rate holds in expectation. 
\end{corollary}

Akin to the adaptive bound, the above inequality can be viewed as a sum of a matrix version of \eqref{eq:isotonic-global} and an error due to estimation of the unknown permutation.

Having stated the main upper bounds, we digress a little to remark that the proofs of Theorem~\ref{thm:adaptive} and Theorem~\ref{thm:global} also yield a minimax optimal rate of estimation (up to logarithmic factors) for unimodal regression, which improves the bound \eqref{eq:unimodal}. We discuss the details in Appendix~\ref{sec:unimodal}.

\subsection{Minimax lower bounds}

Given the model $Y = \Pi^* A^* + Z$ where entries of $Z$ are i.i.d. $N(0, \sigma^2)$ random variables, let $(\hat \Pi, \hat A)$ denote any estimator of $(\Pi^*, A^*)$, i.e., any pair in $\mathfrak S_n \times {\R}^{n\times m}$ that is measurable with respect to the observation $Y$. We will prove lower bounds that match the rates of estimation in Corollary~\ref{cor:adaptive} and Corollary~\ref{cor:global} up to logarithmic factors. The combination of upper and lower bounds, implies simultaneous near optimality of the least squares estimator over a large scale of matrix classes.

\medskip

For $m \le K_0\le nm$ and  $V_0>0$, define
$\mathcal U_{K_0}^m = \big\{A \in \mathcal U^m: K(A) \le K_0 \big\} $
and
$\mathcal U^m(V_0) = \big\{A \in \mathcal U^m: V(A) \le V_0 \big\}. $ We present below two lower bounds, one for the adaptive rate uniformly over $\mathcal U_{K_0}^m$ and one for the global rate uniformly over $\mathcal U^m(V_0)$. This splitting into two cases is solely justified by better readability but it is worth noting that a stronger lower bound that holds on the intersection $\mathcal U_{K_0}^m\cap\mathcal U^m(V_0) $ can also be proved and is presented as Proposition~\ref{prop:lower-stronger}.

\begin{theorem} \label{thm:lower-adaptive}
There exists a constant $c \in (0,1)$ such that for any $K_0 \ge m$, and any estimator $(\hat \Pi, \hat A)$, it holds that
\[ 
\sup_{(\Pi, A)  \in \mathfrak S_n\times \mathcal U_{K_0}^m} {\p}_{\Pi A}\Big[ \frac{1}{nm} \| \hat \Pi\hat A - \Pi A\|_F^2\gtrsim  \sigma^2 \Big( \frac{K_0}{nm} + \frac {\log l}m \Big) \Big] 
\ge c ,
\]
where $l = \min(K_0 - m, m) + 1$ and ${\p}_{\Pi A}$ is the probability distribution of $Y= \Pi A+ Z$. 
It follows that the lower bound with the same rate holds in expectation.
\end{theorem}

In fact, the lower bound holds for any estimator of the matrix $\Pi^* A^*$, not only those of the form $\hat \Pi \hat A$ with $\hat A \in \mathcal U^m$. The above lower bound matches the upper bound in Corollary~\ref{cor:adaptive} up to logarithmic factors. 

Note the presence of a $\log l$ factor in the second term. If $l = 1$ then $K_0 = m$ which means that each column of $A$ is simply a constant block, so $\Pi A = A$ for any $\Pi \in \mathfrak S_n$. In this case, the second term vanishes because the permutation does not play a role. More generally, the number $l-1$ can be understood as the maximal number of columns of $A$ on which the permutation does have an effect. The larger $l$, the harder the estimation. 
It is easy to check that if $l \ge n$ the second term in the lower bound will be dominated by the first term in the upper bound.

A lower bound corresponding to Corollary~\ref{cor:global} also holds:

\begin{theorem} \label{thm:lower-global}
There exists a constant $c \in (0,1)$ such that for any $V_0 \ge 0$,  and any estimator $(\hat \Pi, \hat A)$, it holds that
$$
\sup_{(\Pi, A) \in \mathfrak S_n\times \mathcal U^m(V_0)} {\p}_{\Pi A}\Big[ \frac{1}{nm} \| \hat \Pi\hat A - \Pi A\|_F^2\gtrsim \big(\frac{\sigma^2 V_0}{n}\big)^{2/3}  +  \frac {\sigma^2}n  
+ \frac {\sigma^2}m\wedge  m^2 V_0^2  \Big] 
\ge c \,,
$$
where ${\p}_{\Pi A}$ is the probability distribution of $Y= \Pi A+ Z$. 
The lower bound with the same rate also holds in expectation.
\end{theorem}

There is a slight mismatch between the upper bound of Corollary~\ref{cor:global} and the lower bound of Theorem~\ref{thm:lower-global} above. Indeed the lower bound features a term $\frac {\sigma^2} m \wedge  m^2 V_0^2$ instead of just $\frac{\sigma^2}m$. In the regime $m^2 V_0^2 < \frac{\sigma^2}m$, where $A$ has very small variation, the LS estimator may not be optimal. Proposition~\ref{prop:special-regime} indicates that a matrix with constant columns obtained by averaging achieves optimality in this extreme regime.


\section{Further results in the monotone case}
\label{sec:rankscore}
A particularly interesting subset of unimodal matrices is $\mathcal S^m$, the set of $n\times m$ matrices with monotonically increasing columns. While it does not amount to the seriation problem in its full generality, this special case is of prime importance in the context of shape constrained estimation as illustrated by the  discussion and references in Section~\ref{sec:related-work}. In fact, it covers the example of bipartite ranking discussed at the end of Section~\ref{sec:related-work}. In the rest of this section, we devote further investigation to this important case. To that end, consider the model \eqref{eq:model} 
where we further assume that $A^* \in \mathcal S^m$. We refer to this model as the \emph{monotone seriation model}.
In this context, define the LS estimator by
\[ (\hat \Pi, \hat A) \in \operatorname*{argmin}_{(\Pi, A) \in \mathfrak S_n \times \mathcal S^m} \|Y - \Pi A\|_F^2 \,. \]
Since  $\mathcal S^m$ is a convex subset of $\mathcal U^m$, it is easily seen that the upper bounds in Theorem~\ref{thm:adaptive} and \ref{thm:global} remain valid in this case.  The lower bounds of Theorem~\ref{thm:lower-adaptive} (with $\log l$ replaced by $1$) and Theorem~\ref{thm:lower-global} also extend to this case; see Appendix~\ref{sec:lower}.

Although for unimodal matrices the established error bounds do not imply any bounds on estimation of $A^*$ or $\Pi^*$ in general, for the monotonic case, however, Lemma~\ref{lem:rearrange} yields that
\[ \|\hat A - A^*\|_F^2\vee \frac{1}{4}\|(\hat \Pi-\Pi^*)A^*\|_F^2 \le \|\hat \Pi \hat A - \Pi^* A^*\|_F^2 \,.\] 
so that the LS estimator $(\hat \Pi, \hat A)$ also leads to good individual estimators of $\Pi^*$ and $A^*$ respectively.

\medskip

Because it requires optimizing over a union of $n!$ cones $\Pi \mathcal S^m$, no  efficient way of computing the  LS estimator is known since. As an alternative, we describe a simple and efficient algorithm to estimate $(\Pi^*, A^*)$ and study its rate of estimation.

%

Let $K(A)$ and $V(A)$ be defined as before. Moreover, for a matrix $A \in \mathcal S^m$, let $\mathcal J$ denote the set of pairs of indices $(i,j) \in [n]^2$ such that $A_{i,\cdot}$ and $A_{j,\cdot}$ are not identical. Define the quantity $R(A)$ by
\begin{equation} \label{eq:r}
R(A) = \frac 1n \max_{\substack{\mathcal I \subset [n]^2\\ |\mathcal I| = n}} \sum_{(i, j) \in \mathcal I \cap \mathcal J} \Big( \frac{\|A_{i,\cdot} - A_{j,\cdot}\|_2^2}{\|A_{i,\cdot} - A_{j,\cdot}\|_\infty^2}\wedge \frac{m \|A_{i,\cdot} - A_{j,\cdot}\|_2^2}{\|A_{i,\cdot} - A_{j,\cdot}\|_1^2} \Big) \,.
\end{equation}
It can be shown (see Appendix~\ref{sec:monotone}) that $1\le R(A) \le \sqrt{m}$. Intuitively, the quantity $R(A)$ is small if the difference $u$ of any two rows of $A$ is either very sparse ($\|u\|_2/\|u\|_\infty$ is small) or very dense ($m\|u\|_2/\|u\|_1$ is small). Indeed, for any nonzero vector $u \in {\R}^m$, $\|u\|_2^2/\|u\|_\infty^2 \ge 1$ with equality achieved when $\|u\|_0 = 1$, and $m\|u\|_2^2/\|u\|_1^2 \ge 1$ with equality achieved when all entries of $u$ are the same. 

For matrices with small $R(\cdot)$ values, it is possible to aggregate the information across each row to learn the unknown permutation $\Pi^*$ in a simple fashion.  Recovering the permutation $\Pi^*$,   is equivalent to ordering (or ranking reversely) the rows of $\Pi^*A^*$ from their noisy version $Y$. 

One simple method to achieve this goal, which we call \RankSum, is to permute the rows of $Y$ so that they have increasing row sums.
However, it is easy to observe that this method fails if
\begin{equation}\label{eq:specialA}
A^* =
\begin{bmatrix}
\sqrt{m}& 0 & \dots & 0 \\
2\sqrt{m} & 0 & \dots & 0 \\
\vdots & \vdots& & \vdots \\
n\sqrt m & 0 & \dots & 0
\end{bmatrix}
\end{equation}
where $A^*_{i,1} = i\sqrt m$ and entries of $Z$ are i.i.d. standard Gaussian variables,
because the sum of noise in a row has order $\sqrt m$ which is no less than the gaps between row sums of $A^*$. In fact, $R(A^*) = 1$ and it should be easy to distinguish the two types of rows of $A^*$, for example, by looking at the first entry of a row. This motivates us to consider the following method called \RankScore. 

For $i, i' \in [n]$, define
\[
\Delta_{A^*}(i, i') = \max_{j \in [m]} (A^*_{i', j} - A^*_{i, j})\vee \frac 1{\sqrt m} \sum_{j=1}^m (A^*_{i', j} - A^*_{i, j}) 
\]
and define $\Delta_Y(i, i')$ analogously. 
The \RankScore\ procedure is defined as follows:
\begin{enumerate}
\item For each $i \in [n]$, define the score $s_i$ of the $i$-th row of $Y$ by
$$
s_i =  \sum_{l=1}^n \1(\Delta_Y(l, i) \ge 2\tau) 
$$
where $\tau := C \sigma \sqrt{\log(nm)}$ for some tuning constant $C$ (see Appendix~\ref{sec:monotone} for more details). 
\item Then order the rows of $Y$ so that their scores are increasing, with ties broken arbitrarily. 
\end{enumerate}

The \RankScore\ procedure recovers an order of the rows of $Y$, which leads to an estimator $\tilde \Pi$ of the permutation. Then we define $\tilde A \in \mathcal S^m$ so that $\tilde \Pi \tilde A$ is the projection of $Y$ onto the convex cone $\tilde \Pi \mathcal S^m$. The estimator $(\tilde \Pi, \tilde A)$
enjoys the following rate of estimation.

\begin{theorem} \label{thm:monotone}
For $A^* \in \mathcal S^m$ and $Y = \Pi^* A^* + Z$, let $(\tilde \Pi, \tilde A)$ be the estimator defined above using the \RankScore\ procedure with threshold $\tau = 3 \sigma \sqrt{(C+1)\log(nm)}$, $C>0$.
Then it holds that
\begin{align*}
\frac{1}{nm} \|\tilde \Pi\tilde A - \Pi^*A^*\|_F^2 
\lesssim \min_{A \in \mathcal S^m} \Big(&\frac 1{nm}\|A - A^*\|_F^2  + \sigma^2 \frac{K(A)}{nm} \log \frac{enm}{K(A)} \Big) \\
&+  (C+1) \sigma^2 \frac{R(A^*) \log(nm)}{m}\,,
\end{align*}
with probability at least $1- e^{-c(n+m)} -(nm)^{-C}$ for some constant $c>0$.
\end{theorem}


The quantity $R(A^*)$ only depends on the matrix $A^*$. If $R(A^*)$ is bounded logarithmically, the estimator $(\tilde \Pi, \tilde A)$ achieves the minimax rate up to logarithmic factors. In any case, $R(A^*) \le \sqrt m$, so the estimator is still consistent with the permutation error (the last term) decaying at a rate no slower than $\tilde O(\frac{1}{\sqrt m})$. Furthermore, it is worth noting that $R(A^*)$ is not needed to construct $(\tilde \Pi, \tilde A)$, so the estimator adapts to $R(A^*)$ automatically.

\begin{remark}
In the same way that Theorem~\ref{thm:global} follows from Theorem~\ref{thm:adaptive}, we can deduce from Theorem~\ref{thm:monotone}
a global bound for the estimator $(\tilde \Pi, \tilde A)$ which has rate
\[ \Big(\frac{\sigma^2 V(A^*)\log n}{n}\Big)^{2/3} + \sigma^2 \Big(\frac{\log n}n + R(A^*)  \frac{\log(nm)}{m} \Big)\,. \]
\end{remark}

We conclude this section with a numerical comparison between the \RankSum\ and \RankScore\ procedures.

Consider the model \eqref{eq:model} with $A^*\in \mathcal S^m$ and assume without loss of generality that $\Pi^*=I_n$. For various $n \times m$ matrices $A^*$, we generate observations $Y=A^*+Z$ where entries of $Z$ are  i.i.d. standard Gaussian variables. 
The performance of the estimators given by \RankScore\ and \RankSum\ defined above is compared to the performance of the oracle $\hat A^{\mathrm{oracle}}$ defined by the projection of $Y$ onto the cone $\mathcal S^m$. 
For the \RankScore\ estimator we take $\tau=6$. The curves are generated based on $30$ equally spaced points on the base-$10$ logarithmic scale, and all results are averaged over $10$ replications. The vertical axis represents the estimation error of an estimator $\hat \Pi\hat A$, measured by the sample mean of  $\log_{10} \big(\frac{1}{nm}\| \hat \Pi \hat A- A^* \|_F^2 \big)$ unless otherwise specified.

\begin{figure}[ht]
\centering
\begin{minipage}[c]{.49\linewidth}
\includegraphics[width=\linewidth]{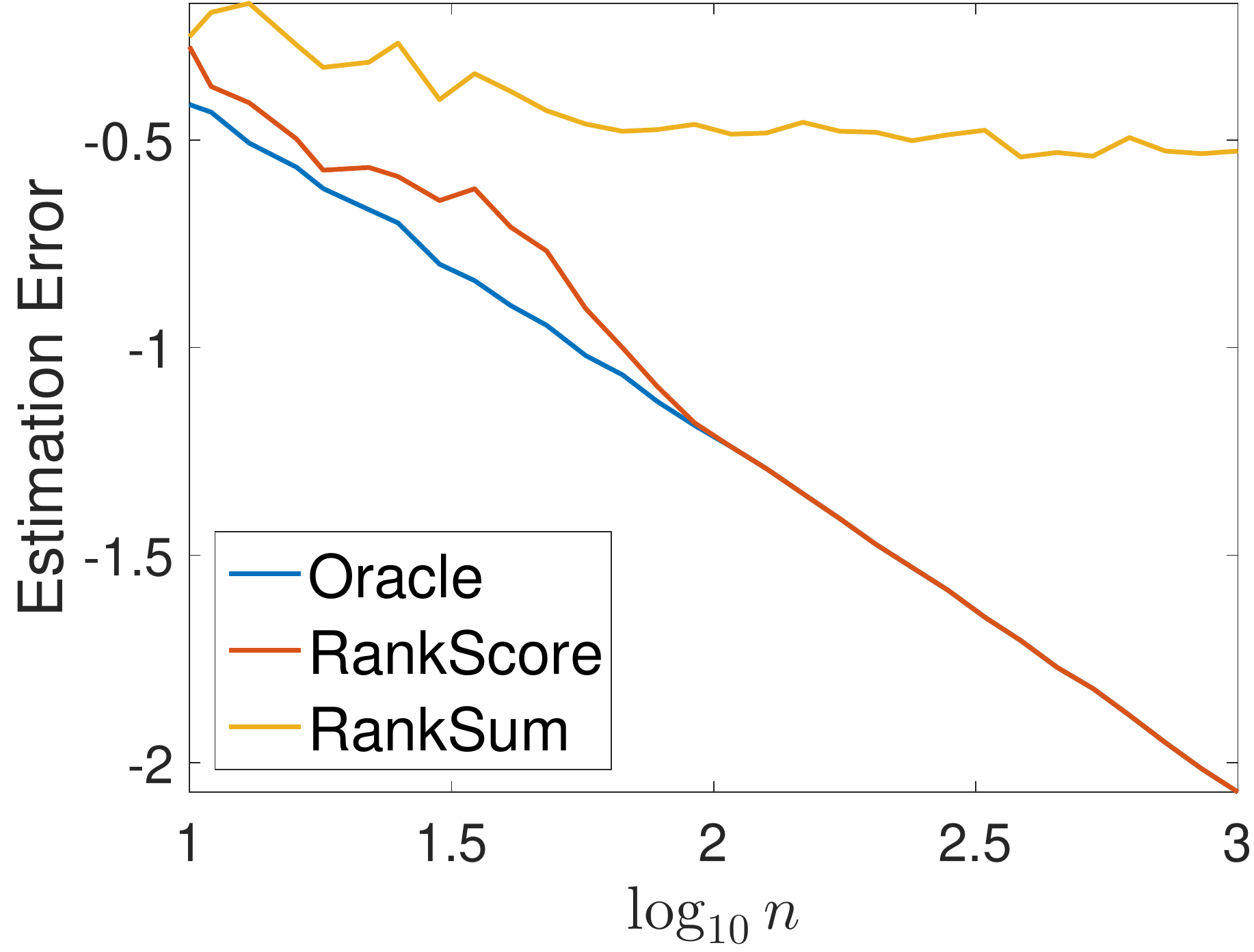}
\end{minipage} 
\begin{minipage}[c]{.49\linewidth}
\includegraphics[width=\linewidth]{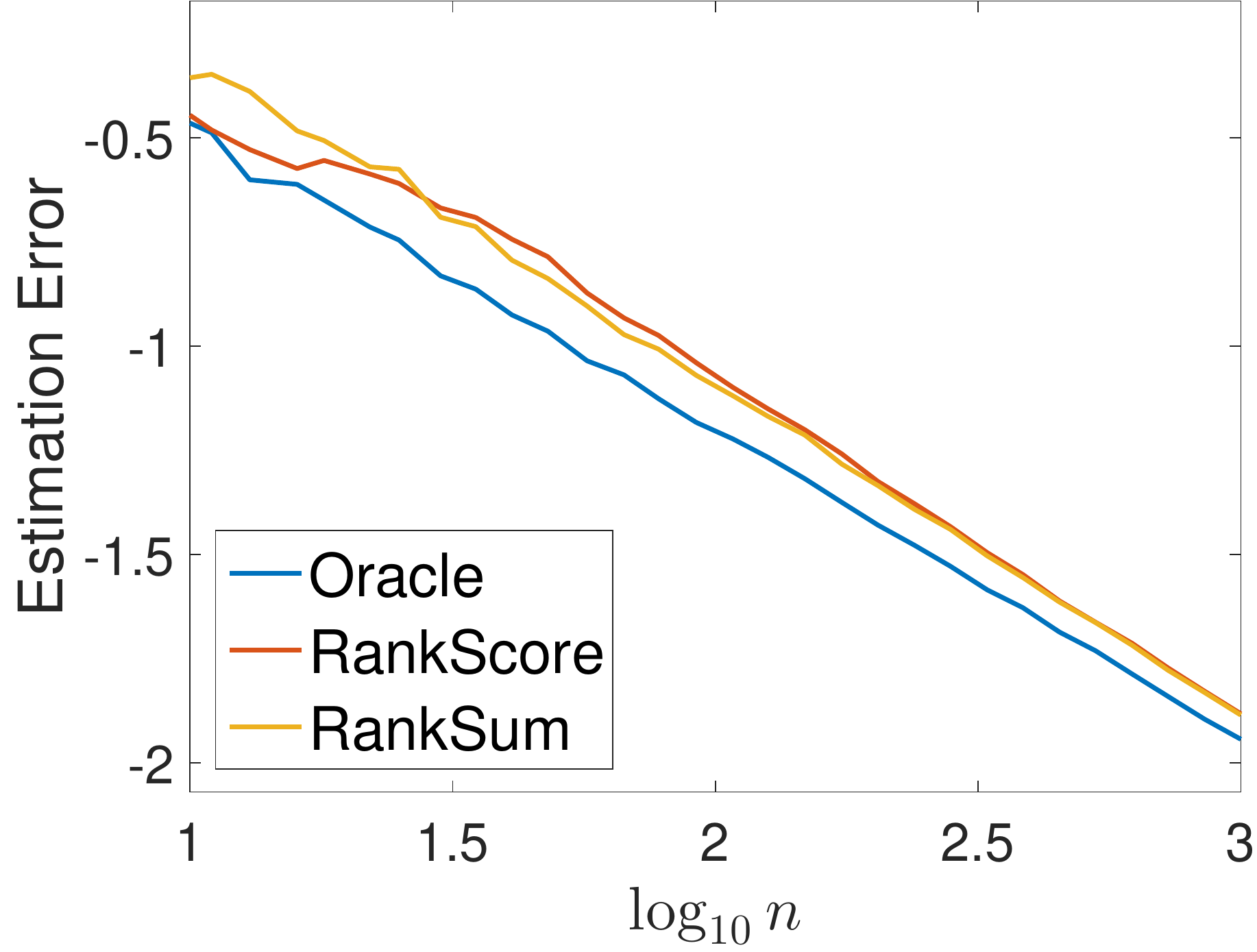}
\end{minipage} 
\caption{Estimation errors of three estimators for two deterministic $A^*$ of size $n\times n$. Left: rows of $A^*$ are $1$-sparse; Right: columns of $A^*$ are identical.}
\label{fig:plot1}
\end{figure}

We begin with two simple examples for which we set $n=m$. 
In the left plot of Figure~\ref{fig:plot1}, $A^*$ is  defined  as in \eqref{eq:specialA}. As expected, \RankSum\ fails to estimate the true permutation and performs very poorly. 
On the other hand, \RankScore\ succeeds in recovering the correct permutation and has roughly the same performance as the oracle. Because the difference of any two rows of $A^*$ is $1$-sparse, $R(A^*) = 1$ according to \eqref{eq:r} and the discussion thereafter. Hence, Theorem~\ref{thm:monotone} predicts the fast rate, which is verified by the experiment. The right plot illustrates another extreme case; more precisely, we set $A^*$ to be the matrix with all $m$ columns equal to $\frac{1}{n}(1,\cdots, n)^\top$. The difference of any two rows of $A^*$ is constant across all entries, so again we have $R(A^*) = 1$ by \eqref{eq:r}. Thus \RankScore\ achieves the fast rate as expected. Note that \RankSum\ also performs well in this case.

\begin{figure}[ht]
\centering
\begin{minipage}[c]{.49\linewidth}
\includegraphics[width=\linewidth]{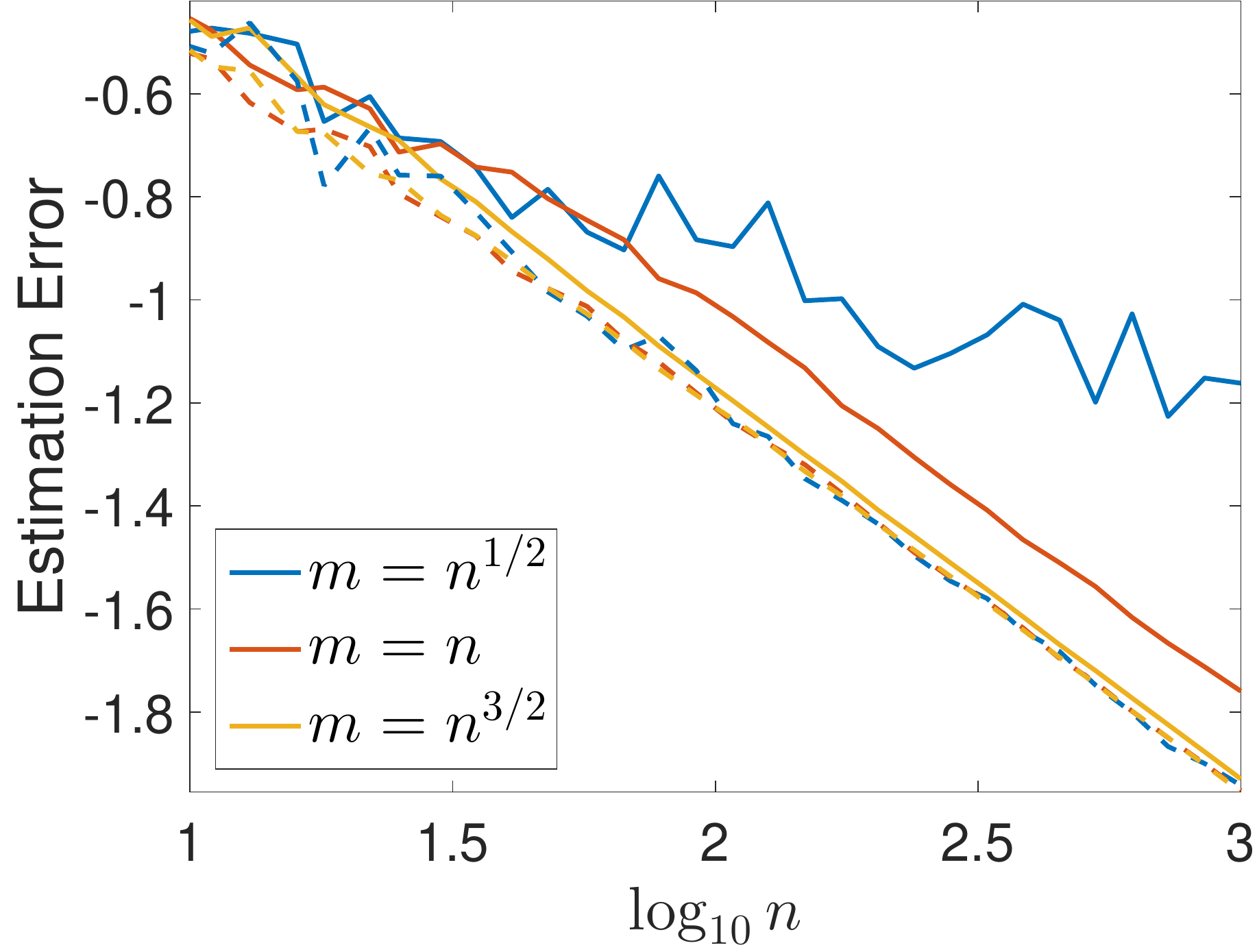}
\end{minipage}
\begin{minipage}[c]{.49\linewidth}
\includegraphics[width=\linewidth]{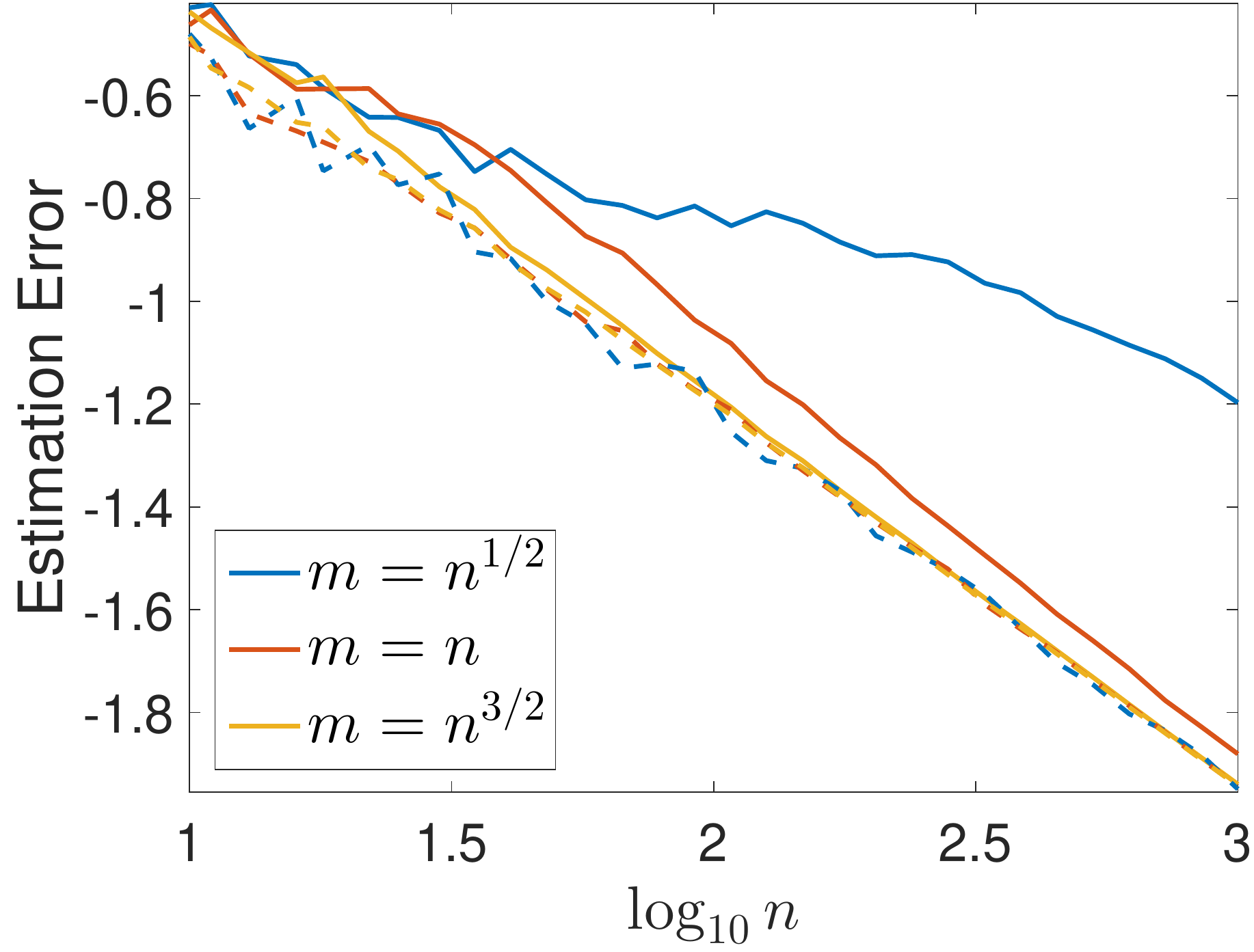}
\end{minipage} 
\caption{Estimation errors of the oracle (dashed lines) and \RankScore\ (solid lines) for different regimes of $(n, m)$ and  randomly generated $A^*$ of size $n \times m$. Left: $K(A^*)=5m$; Right: $V(A^*) \le 1$.}
\label{fig:plot2}
\end{figure}

In Figure~\ref{fig:plot2}, we compare the performance of \RankScore\ to that of the oracle in three regimes of $(n, m)$. The matrices $A^*$ are randomly generated for different values of $n$ and $m$ as follows. 
For the right plot, $A^*$ is generated so that $V(A^*) \le 1$, by sorting the columns of a matrix with i.i.d. $U(0,1)$ entries. 
For the left plot, we further require that $K(A^*)=5m$ by uniformly partitioning each column of $A^*$ into five blocks and assigning each block the corresponding value from a sorted sample of five i.i.d. $U(0,1)$ variables.

Since the oracle knows the true permutation, its behavior is independent of $m$, and its rates of estimation are bounded by $\frac{\log n}{n}$ for $K(A^*) = 5m$ and $(\frac{\log n}{n})^{\frac 23}$ for $V(A^*) = 1$ respectively by Theorem~\ref{thm:adaptive} and \ref{thm:global}. (The difference is minor in the plots as $n$ is not sufficiently large). For \RankScore, the permutation term dominates the estimation term when $m=n^{1/2}$ by Theorem~\ref{thm:monotone}. From the plots, the rates of estimation are better than $\tilde O(n^{-1/4})$ predicted by the worst-case analysis in both examples. For $m=n$, we also observe rates of estimation faster than the worst-case rate $\tilde O(n^{-1/2})$ and close to the oracle rates.
We could explain this phenomenon by $R(A^*)<\sqrt m$, but such an interpretation may not be optimal since our analysis is based on worst-case deterministic $A^*$. Potential study of random designs of $A^*$ is left open. 
Finally, for $m=n^{3/2}$, the permutation term is of order $\tilde O(n^{-3/4})$ theoretically, in between of the oracle rates for the two cases. Indeed \RankScore\ has almost the same performance as the oracle experimentally. Overall Figure~\ref{fig:plot2} illustrates the good behavior of \RankScore\ in this random scenario.

\medskip

\begin{figure}[ht]
\centering
\includegraphics[width=.6\linewidth]{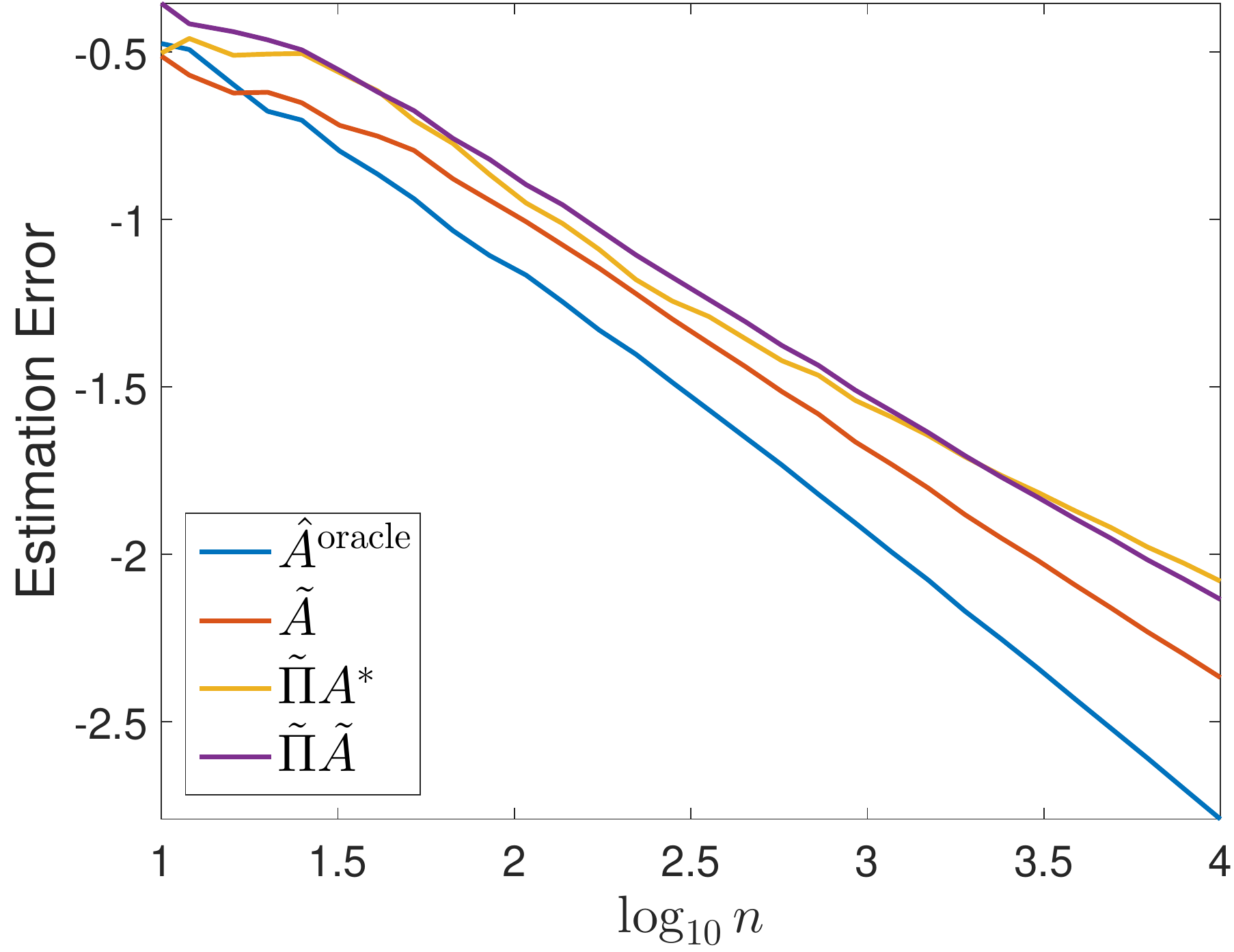}  
\caption{Various estimation errors of the oracle and \RankScore\ for the triangular matrix.}
\label{fig:plot3}
\end{figure}

To conclude our numerical experiments, we consider the $n \times n$ lower triangular matrix $A^*$ defined by $A^*_{i,j}=\1(i \ge j)$. For this matrix, it is easy to check that
$K(A^*)=2n-1$
and $R(A^*)\approx \sqrt n$. We plot  in Figure~\ref{fig:plot3} the estimation errors of $\tilde \Pi \tilde A$, $\tilde \Pi A^*$ and $\tilde A$ given by \RankScore, in addition to the oracle. By Theorem~\ref{thm:monotone}, the rate of estimation achieved by $\tilde \Pi \tilde A$ is of order $\tilde O(n^{-1/2})$, while that achieved by the oracle is of order $\tilde O(n^{-1})$ since there is no permutation term. The plot confirms this discrepancy.
Moreover, $\frac{1}{n^2}\|\tilde \Pi A^* - A^*\|_F^2$ is an appropriate measure of the performance of $\tilde \Pi$ by Lemma~\ref{lem:approx-est} and \ref{lem:rearrange}, and the plot suggests that the rates of estimation achieved by $\tilde \Pi A^*$ and  $\tilde \Pi\tilde A$ are about the same order.
Finally $\tilde A$ seems to have a slightly faster rate of estimation than $\tilde \Pi \tilde A$, so in practice $\tilde A$ could be used to estimate $A$. 
However we refrain from making an explicit conjecture about the rate.

\section{Discussion} \label{sec:discussion}

While computational aspects of the seriation problem have received significant attention, the robustness of this problem to noise was still unknown to date. To overcome this limitation, we have introduced in this paper the statistical seriation model and studied optimal rates of estimation by showing, in particular, that the least squares estimator enjoys several desirable statistical properties such as adaptivity and  minimax optimality (up to logarithmic terms). 

While this work paints a fairly complete statistical picture of the statistical seriation model, it also leaves many unanswered questions. There are several logarithmic gaps in the bounds. In the case of adaptive bounds, some logarithmic terms are unavoidable as illustrated by Theorem~\ref{thm:lower-adaptive} (for the permutation term) and also by statistical dimension consideration explained in \cite{Bel15} (for the estimation term). However, a more refined argument for the uniform bound, namely one that uses covering in $\ell_2$-norm rather than $\ell_\infty$-norm, would allow us to remove the $\log n$ factor from the estimation term in the upper bound of Corollary~\ref{cor:global}. Such an argument can be found in~\cite{BirSol67,AnuBabGod79,Gee91} for the larger class of vectors with bounded total variation (see~\cite{MamGee97}) but we do not pursue sharp logarithmic terms in this work. For the permutation term, $\log n$ in the upper bound of Corollary~\ref{cor:adaptive} and $\log l$ in the lower bound of Theorem~\ref{thm:lower-adaptive} do not match if $l < n$. We do not seek answers to these questions in this paper but note that their answers may be different for the unimodal and the monotone case.

Perhaps the most pressing question is that of computationally efficient estimators. Indeed, while statistically optimal, the least squares estimator requires searching through $n!$ permutations, which is not realistic even for problems of moderate size, let alone genomics applications. We gave a partial answer to this question in the specific context of monotone columns by proposing and studying the performance of a simple and efficient estimator called \RankScore. This study reveals the existence of a potentially intrinsic gap between the statistical performance achievable by efficient estimators and that achievable by estimators with access to unbounded computation. A similar gap is also observed in the SST model for pairwise comparisons~\cite{ShaBalGunWai15}. We conjecture that achieving optimal rates of estimation in the seriation model is computationally hard in general but argue that the planted clique assumption that has been successfully used to establish statistical vs. computational gaps in~\cite{
BerRig13b, MaWu15, ShaBalWai16} for example, is not the correct primitive. Instead, one has to seek for a primitive where hardness comes from searching through permutations rather than subsets. 


\section*{Acknowledgments}
\label{sec:acknowledge}
P.R. is partially supported by  NSF grant DMS-1317308 and NSF CAREER Award DMS-1053987. N.F. gratefully acknowledges partial support by NSF grant DMS-1317308, the Institute for Data Systems and Society and the Mathematics Department during his visit at MIT. C.M. acknowledges partial support by NSF CAREER Award DMS-1053987. P.R. also thanks Enno Mammen for his help in retracing the literature on sharp entropy bounds for monotone classes and Ramon van Handel for a stimulating discussion on entropy numbers. The authors thank Alexandre Tsybakov for bringing concurrent work on unimodal regression by Pierre C. Bellec to their attention. Finally, the authors thank Pierre C. Bellec for pointing out an error in an earlier version of Lemma~A.1 and suggesting that it holds for all closed sets.

\appendix

\section{Proof of the upper bounds} \label{sec:upper}

Before proving the main theorems, we discuss two methods adopted in recent works to bound the error of the LS estimator in shape constrained regression, in a general setting. Consider the least squares estimator $\hat \theta$ of the model $y = \theta^* + z$, where $\theta^*$ lies in a parameter space $\Theta$ and $z$ is Gaussian noise. One way to study ${\E}\|\hat \theta - \theta^*\|_2^2$ is to use the \emph{statistical dimension} \cite{Ameetal14} of a convex cone $\Theta$ defined by
\[
{\E}\Big[\Big(\sup_{\theta \in \Theta,\, \|\theta\|_2\le 1}  \langle \theta, z\rangle\Big)^2 \Big] \,.
\]
 This has been successfully applied to isotonic and more general shape constrained regression \cite{ChaGunSen15, Bel15}. 

Another prominent approach is to express the error of the LS estimator via what is known as \emph{Chatterjee's variational formula}, proved in~\cite{Cha14} and given by
\begin{equation}
\label{EQ:chatterjeeVF}
 \|\hat \theta - \theta^*\|_2 = \operatorname*{argmax}_{t \ge 0} \Big( \sup_{\theta \in \Theta, \|\theta- \theta^*\|_2\le t} \langle \theta - \theta^*, z\rangle - \frac {t^2}2 \Big) \,. 
\end{equation}
Note that the first term is related to the \emph{Gaussian width} (see, e.g., \cite{ChaRecParWil12}) of $\Theta$ defined by
${\E}[\sup_{\theta \in \Theta} \langle \theta, z\rangle],$
whose connection to the statistical dimension was studied in \cite{Ameetal14}.
The variational formula was first proposed for convex regression \cite{Cha14}, and later exploited in several different settings, including matrix estimation with shape constraints \cite{ChaGunSen15-2} and unimodal regression \cite{ChaLaf15}. Similar ideas have appeared in other works, for example, analysis of empirical risk minimization \cite{Men15}, ranking from pairwise comparison \cite{ShaBalGunWai15} and isotonic regression \cite{Bel15}. In this latter work, Bellec has used the statistical dimension approach to prove spectacularly sharp oracle inequalities that seem to be currently out of reach for methods based on Chatterjee's variational formula~\eqref{EQ:chatterjeeVF}. On the other hand, Chatterjee's variational formula seems more flexible as computations of the statistical dimension based on \cite{Ameetal14} are currently limited to convex sets $\Theta$ with a polyhedral structure. In this paper, we use exclusively Chatterjee's variational formula.




%
%

\subsection{A variational formula for the error of the LS estimator}

We begin the proof by stating an extension of Chatterjee's variational formula. While we only need this lemma to hold for a union of closed convex sets we present a version that holds for all closed sets. The latter extension was suggested to us by Pierre C. Bellec in a private communication~\cite{Bel16}.

\begin{lemma} \label{lem:variational}
Let $\mathcal C$ be a closed subset of ${\R}^d$. Suppose that $y = a^* + z$ where $a^* \in \mathcal C$ and  $z \in {\R}^d$. Let $\hat a \in \operatorname*{argmin}_{a \in \mathcal C} \|y- a\|_2^2$ be a projection of $y$ onto $\mathcal{C}$. Define the function $f_{a^*}: {\R}_+ \to {\R}$ by
\[
f_{a^*}(t) = \sup_{a \in \mathcal C\cap \mathcal B^d(a^*,t)} \langle  a - a^*, z\rangle - \frac{t^2}2 \,.
\]
Then we have
\begin{equation} \label{eq:variational}  
\|\hat a - a^*\|_2 \in \operatorname*{argmax}_{t \ge 0} f_{a^*}(t). 
\end{equation}
Moreover, if there exists $t^* > 0$ such that $f_{a^*}(t) < 0$ for all $t \ge t^*$, then
$\|\hat a - a^*\|_2 \le t^*.$
\end{lemma}

\begin{proof}
By definition, 
\begin{align*}
\hat a &\in 
\operatorname*{argmin}_{a \in \mathcal C} \Big( \|a - a^*\|_2^2 - 2\langle a - a^*, z\rangle  + \|z\|_2^2 \Big) \\
&= 
\operatorname*{argmax}_{a \in \mathcal C} \Big( \langle a - a^*, z\rangle - \frac 12 \|a - a^*\|_2^2 \Big)
\,.
\end{align*}
Together with the definition of $f_{a^*}$, this implies that
\begin{align*} 
f_{a^*}(\|\hat a - a^*\|_2) &\ge \langle \hat a - a^*, z\rangle - \frac 12 \|\hat a - a^*\|_2^2  \\
& \ge \sup_{a \in \mathcal C\cap \mathcal B^d(a^*,t)} \Big(\langle  a - a^*, z\rangle - \frac 12 \|a - a^*\|_2^2 \Big) \\
& \ge \sup_{a \in \mathcal C\cap \mathcal B^d(a^*,t)} \langle  a - a^*, z\rangle - \frac {t^2}2  = f_{a^*}(t) \,.
\end{align*}
Therefore \eqref{eq:variational} follows.

Furthermore, suppose that there is $t^* > 0$ such that $f_{a^*}(t) < 0$ for all $t \ge t^*$. Since $f_{a^*}(\|\hat a - a^*\|_2) \ge f_{a^*}(0) = 0$, we have $\|\hat a - a^*\|_2 \le t^*$.
\end{proof}

Note that this structural result holds for any error vector $z \in \R^d$ and any closed set $\mathcal{C}$ which is not necessarily convex. In particular, this extends the results in \cite{Cha14} and \cite{ChaLaf15} which hold for convex sets and finite unions of convex sets respectively.

\subsection{Proof of Theorem~\ref{thm:adaptive}}

For our purpose, we need a standard chaining bound on the supremum of a sub-Gaussian process that holds in high probability. The interested readers can find the proof, for example, in \cite[Theorem~5.29]{van14}, and refer to \cite{LedTal91} for a more detailed account of the technique.

\begin{lemma}[Chaining tail inequality] \label{lem:chaining-tail}
Let $\Theta \subset {\R}^d$ and $z \sim \operatorname{subG}(\sigma^2)$ in ${\R}^d$. For any $\theta_0 \in \Theta$, it holds that
\[ \sup_{\theta \in \Theta} \langle \theta - \theta_0, z\rangle \le C \sigma \int_0^{\operatorname{diam}(\Theta)} \sqrt{\log{ N(\Theta, \|\cdot\|_2, \varepsilon)}} \, d \varepsilon + s \]
with probability at least $1-C \exp(-\frac{c s^2}{\sigma^2 \operatorname{diam}(\Theta)^2})$ where $C$ and $c$ are positive constants.
\end{lemma}

Let $\tilde A \in \mathcal U^m$. To ligthen the notation, we define two rates of estimation:
\begin{equation}\label{eq:r1} 
R_1 = R_1(\tilde A, n) = \sigma \Big(\sqrt {K(\tilde A) \log \frac{e n m}{K(\tilde A)} } + \sqrt{n \log n}\Big) 
\end{equation}
and
\begin{equation} \label{eq:r2}
R_2 = R_2(\tilde A, n)= \sigma^2 \Big(K(\tilde A) \log \frac{e n m}{K(\tilde A)} + n  \log n\Big). 
\end{equation}
Note that $R_2 \le R_1^2 \le 2 R_2$.

\begin{lemma} \label{lem:f-bound}
Suppose $Y = A^* + Z$ where $A^* \in {\R}^{n\times m}$ and $Z \sim \operatorname{subG}(\sigma^2)$. For $\tilde A \in \mathcal U^m$ and all  $t > 0$, define
\[
f_{\tilde A}(t) = \sup_{A \in \mathcal M \cap \mathcal B^{nm}(\tilde A, t)} \langle A - \tilde A, Y- \tilde A\rangle - \frac{t^2}2. 
\]
Then for any $s>0$, it holds simultaneously for all $t > 0$ that
\begin{equation} \label{eq:f-bound}
f_{\tilde A}(t) \le C R_1 t + t \|A^*-\tilde A\|_F  - \frac {t^2}2 + st
\end{equation}
with probability at least $1- C \exp(- \frac{c s^2}{\sigma^2})$, where $C$ and $c$ are positive constants.
\end{lemma}

\begin{proof}
Define $\Theta = \Theta_{\mathcal M}(\tilde A, 1) = \bigcup_{\lambda \ge 0} \{B - \lambda \tilde A: B \in \mathcal M \cap \mathcal B^{nm}(\lambda \tilde A, 1) \}$ (see also Definition~\eqref{eq:def-theta}). In particular, $\Theta \subset \mathcal B^{nm}(0,1)$ and $0 \in \Theta$.
Since $\mathcal M$ is a finite union of convex cones and thus is star-shaped, by scaling invariance,
\[
\sup_{A \in \mathcal M \cap \mathcal B^{nm}(\tilde A, t)} \langle A - \tilde A, Z\rangle 
= t \sup_{B \in \mathcal M \cap \mathcal B^{nm}(t^{-1} \tilde A, 1)} \langle B - t^{-1}\tilde A, Z\rangle 
\le t \sup_{M \in \Theta} \langle M, Z\rangle .
\]
By Lemma~\ref{lem:chaining-tail}, with probability at least $1- C \exp(- \frac{c s^2}{\sigma^2})$,
\[
\sup_{M \in \Theta} \langle M, Z\rangle \le C \sigma \int_0^2 \sqrt{\log{ N(\Theta, \|\cdot\|_F, \varepsilon)}} \, d \varepsilon + s \,.
\]
Moreover, it follows from Lemma~\ref{lem:cover-matrix-union} that 
\[ \log{ N(\Theta, \|\cdot\|_F, \varepsilon)}\leq C \varepsilon^{-1} K(\tilde A)  \log  \frac{enm}{K(\tilde A)} + n  \log n \,. \]
Combining the previous three displays, we see that
\begin{align*} 
\sup_{A \in \mathcal M \cap \mathcal B^{nm}(\tilde A, t)} \langle A - \tilde A, Z\rangle &\le C \sigma t \int_0^{2} \sqrt{C \varepsilon^{-1} K(\tilde A)  \log  \frac{enm}{K(\tilde A)} + n  \log n } \,\, d \varepsilon + st \\
& \le C \sigma t \sqrt {K(\tilde A) \log \frac{enm}{K(\tilde A)}} 
+ C \sigma t \sqrt{ n  \log n} + st \\
& = C R_1 t + st 
\end{align*}
with probability at least $1- C \exp(- \frac{c s^2}{\sigma^2})$. Therefore
\begin{align*} 
f_{\tilde A}(t) &= \sup_{A \in \mathcal M \cap \mathcal B^{nm}(\tilde A, t)} \langle  A- \tilde A, Y-\tilde A \rangle - \frac{t^2}2 \\
&\le \sup_{A \in \mathcal M \cap \mathcal B^{nm}(\tilde A, t)} \langle  A- \tilde A, Z \rangle +  \sup_{A \in \mathcal M \cap \mathcal B^{nm}(\tilde A, t)}  \langle  A- \tilde A, A^*-\tilde A \rangle - \frac{t^2}2 \\
&\le C R_1 t  + st + t \|A^*-\tilde A\|_F  - \frac {t^2}2 
\end{align*}
with probability at least $1- C \exp(- \frac{c s^2}{\sigma^2})$ simultaneously for all $t > 0$.
\end{proof}

We are now in a position to prove the adaptive oracle inequalities in Theorem~\ref{thm:adaptive}. 
Recall that $(\hat \Pi, \hat A)$ denotes the LS estimator defined in \eqref{def-pi-a}.  Without loss of generality, assume that $\Pi^* = I_n$ and $Y = A^* + Z$.

Fix $\tilde A \in \mathcal U^m$ and define $f_{\tilde A}$ as in Lemma~\ref{lem:f-bound}. We can apply Lemma~\ref{lem:variational} with $a^* = \tilde A$, $z = Y - \tilde A$, $y = Y$ and $\hat a = \hat \Pi \hat A$ to achieve an error bound on $\|\hat \Pi \hat A - \tilde A\|_F,$ since $\hat \Pi \hat A \in \operatorname*{argmin}_{M \in \mathcal M} \|Y - M\|_F^2$. To be more precise, for any $s>0$ we define $t^* = 3C_1 R_1 + 2 \|A^* - \tilde A\|_F + 2s$ where $C_1$ is the constant in \eqref{eq:f-bound}. Then it follows from Lemma~\ref{lem:f-bound} that with probability at least $1- C\exp(-\frac{c s^2}{\sigma^2 })$, it holds for all $t  \ge t^*$ that
\[ f_{\tilde A}(t) \le C_1 R_1 t + t \|A^*-\tilde A\|_F  - \frac {t^2}2 + st < 0 \,. \]
Therefore by Lemma~\ref{lem:variational}, 
\[
\|\hat \Pi \hat A - \tilde A\|_F \le t^* = 3C_1 R_1 + 2 \|A^* - \tilde A\|_F + 2s \,,
\]
and thus
\begin{equation} \label{eq:adaptive-frobenius} 
\|\hat \Pi \hat A - A^*\|_F \le C (R_1 + \|A^* - \tilde A\|_F) + 2s
\end{equation}
with probability at least $1- C\exp(-\frac{c s^2}{\sigma^2}).$

In particular,
if $s = R_1$, then $s \ge \sigma \sqrt{n+m}$ as $K(\tilde A) \ge m$. We see that with probability at least $1- C \exp(-\frac{c s^2}{\sigma^2}) \ge 1- e^{-c(n+m)}$,
\[ 
\|\hat \Pi \hat A - A^*\|_F \lesssim R_1 + \|A^* - \tilde A\|_F
\]
and thus
\[
\|\hat \Pi \hat A - A^*\|_F^2 \lesssim   \|A^* - \tilde A\|_F^2 +  \sigma^2  K(\tilde A)\log \frac{enm}{K(\tilde A)} + \sigma^2 n\log n \,.
\]
Finally,  \eqref{eq:oracle-loss} follows by taking the infimum over $\tilde A \in \mathcal U^m$ on the right-hand side and dividing both sides by $nm$.

Next, to prove the bound in expectation, observe that \eqref{eq:adaptive-frobenius} yields
\[{\p}\Big[\|\hat \Pi \hat A - A^*\|_F^2 - C (R_2 + \|A^*-\tilde A\|_F^2) \ge s \Big] \le C \exp(- \frac{c s}{\sigma^2}),
\]
where $R_2$ is defined in \eqref{eq:r2}.
Integrating the tail probability, we get that
\begin{align*}
{\E}\|\hat \Pi \hat A - A^*\|_F^2 - C (R_2 + \|A^*-\tilde A\|_F^2) &\lesssim  \int_0^\infty \exp(- \frac{c s}{\sigma^2}) \, d s = \frac{\sigma^2}c
\end{align*}
and therefore
\[
{\E}\|\hat \Pi \hat A - A^*\|_F^2 \lesssim R_2 + \|A^*-\tilde A\|_F^2 \,.
\]
Dividing both sides by $nm$ and minimizing over $\tilde A \in \mathcal U^m$ yields \eqref{eq:oracle-mse}.

\subsection{Proof of Theorem~\ref{thm:global}}
In the setting of isotonic regression, \cite{BelTsy15} derived global bounds from adaptive bounds by a block approximation method, which also applies to our setting.  
For $k \in [n]$, let
\[ \mathcal U_{k} = \big\{a \in \mathcal U: \card (\{a_1, \dots, a_n\}) \le k \big\}. \]
Define $k^* = \big\lceil \big( \frac{V(a)^2 n}{\sigma^2 \log(en) } \big)^{1/3}  \big\rceil$. The lemma below is very similar to \cite[Lemma~2]{BelTsy15} and their proof also extends to the unimodal case with minor modifications. We present the result with proof for completeness.

\begin{lemma} \label{lem:vector-approximation}
For $a \in \mathcal U$
and $k \in [n]$,
there exists $\tilde a \in \mathcal U_{k}$ such that 
\begin{equation}\label{eq:vector-approx}
\frac 1{\sqrt n} \|\tilde a - a\|_2 \le \frac{V(a)}{2k}. 
\end{equation}
In particular,
there exists $\tilde a \in \mathcal U_{k^*}$ such that
\[ \frac 1n \|\tilde a - a\|_2^2 \le \frac 14 \max\Big( \Big(\frac{ \sigma^2 V(a) \log (en) }{n}\Big)^{2/3} , \frac{\sigma^2 \log (en)}{n} \Big). \]
Moreover,
\[ \frac{\sigma^2 k^*}{n} \log (en) \le 2 \max\Big( \Big(\frac{\sigma^2 V(a)\log  (en)}{n}\Big)^{2/3} , \frac{\sigma^2 \log (en)}{n} \Big). \]
\end{lemma}

\begin{proof}
Let $\underline a=\min(a_1,a_n)$, $\bar a=\max_{i\in[n]} a_i$ and $i_{0}\in \operatorname*{argmax}_{i\in[n]} a_i$. For $j\in[k-1]$, consider the intervals 
\[ I_j = \Big[\underline a+ \frac{j-1}{k}V(a),  \underline a+ \frac{j}{k}V(a)\Big] , \]
and $ I_k = \Big[\underline a+ \frac{k-1}{k}V(a),  \bar a\Big]$. Also for $j\in [k]$, let $ J_j=\{i\in [n] : a_i\in I_j\}$.
We define the vector  $\tilde a\in {\R}^n$ by $\tilde a_i=\underline a+ \frac{j-1/2}{k}V(a)$ for $i \in [n]$, where $j$ is uniquely determined by $i\in I_j$.
Since $a$ is increasing on $\{1, \dots, i_0\}$ and decreasing $\{i_0, \dots, n\}$, so is 
$\tilde a$.
Thus $\tilde a \in \mathcal U_k$.
Moreover, 
$|\tilde a_i-a_i|\leq \frac{V(a)}{2k}$ for $i\in[n]$, which implies \eqref{eq:vector-approx}.

Next we prove the latter two assertions. Since $k^* = \lceil \big( \frac{V(a)^2 n}{\sigma^2 \log (en)} \big)^{1/3} \rceil$, if $\tilde a \in \mathcal U_{k^*}$ and $k^* = 1$ then
 \[ \frac 1n \|\tilde a - a\|_2^2 \le  \frac{V(a)^2}4 \le \frac{\sigma^2}{4n} \log (en) \]
 and
 \[ \frac{\sigma^2 k^*}{n} \log (en) = \frac{\sigma^2}{n} \log (en). \]
 On the other hand, if $k^* > 1$, then
 \[  \frac 1n \|\tilde a - a\|_2^2 \le \frac{V(a)}{4 (k^*)^2} \le  \frac 14 \big(\frac{\sigma^2 V(a) \log (en)}{n}\big)^{2/3} \]
 and
 \[ \frac{\sigma^2 k^*}{n} \log (en) \le 2 \big(\frac{\sigma^2 V(a) \log (en)}{n}\big)^{2/3}.
 \]
 \end{proof}

It is straightforward to generalize the lemma to matrices. 
For $\mathbf k \in [n]^m$, we write $\mathbf k = (k_1, \dots, k_m)$ and let
\[\mathcal U_\mathbf k^m = \{A \in \mathcal U^m: \card(\{A_{1, j}, \dots, A_{n, j}\}) = k_j \text{ for } 1\le j \le m\} . \]
Then $K(A) = \sum_{j=1}^m k_j$ for $A \in \mathcal U^m_{\mathbf k}$.
Define $\mathbf k^*$ by 
\[ k_j^* = \Big\lceil \Big( \frac{V(A_{\cdot,j})^2 n}{\sigma^2 \log (en)} \Big)^{1/3} \Big\rceil.
\]

\begin{lemma} \label{lem:block-approximation}
For $A \in \mathcal U^m$, there exists $\tilde A \in \mathcal U_{\mathbf k^*}^m$ such that
\[
\frac 1{nm} \|\tilde A - A\|_F^2 \le \frac 14 \Big(\frac{\sigma^2 V(A)\log  (en)}{n}  \Big)^{2/3} + \frac{\sigma^2}{4n}\log (en)  \]
and
\[ \frac{\sigma^2 K(\tilde A)}{nm} \log (en) \le 2 \Big(\frac{\sigma^2 V(A)\log  (en)}{n}  \Big)^{2/3} + \frac{2\sigma^2}{n}\log (en) \,  .\]
\end{lemma}

\begin{proof}
Applying Lemma~\ref{lem:vector-approximation} to columns of $A$, we see that
there exists $\tilde A \in \mathcal U_{\mathbf k^*}^m$ such that
\[
\frac 1n \|\tilde A_{\cdot,j} - A_{\cdot,j}\|_2^2 \le \frac 14 \max\Big( \big(\frac{\sigma^2 V(A_{\cdot,j}) \log  (en)}{n}\big)^{2/3} , \frac{\sigma^2 }{n} \log  (en) \Big) 
\]
and
\[ 
\frac{\sigma^2 k_j^*}{n} \log  (en)  \le 2 \max\Big( \big(\frac{\sigma^2 V(A_{\cdot,j})\log (en)}{n}\big)^{2/3} , \frac{\sigma^2 }{n} \log  (en) \Big). \]
Summing over $1 \le j \le m$, we get that
\begin{align*}
\frac 1{nm} \|\tilde A - A\|_F^2 &\le \frac{1}{4m} \Big(\frac{\sigma^2 \log (en)}{n}\Big)^{2/3} \sum_{j=1}^m V(A_{\cdot,j})^{2/3} + \frac{\sigma^2 \log (en)}{4 n} \\
&=  \frac 14 \Big(\frac{\sigma^2 V(A)\log  (en)}{n}  \Big)^{2/3} + \frac{\sigma^2}{4n}\log (en) \, ,
\end{align*}
and similarly
\[ \frac{\sigma^2 K(\tilde A)}{nm} \log (en) \le 2 \Big(\frac{\sigma^2 V(A)\log  (en)}{n}  \Big)^{2/3} + \frac{2 \sigma^2}{n}\log (en) \, .\]
\end{proof}

For $A \in \mathcal U^m$, choose $\tilde A \in \mathcal U_{\mathbf k^*}^m$ according to Lemma~\ref{lem:block-approximation}. Then
\begin{align} 
\frac 1{nm} \| \tilde  A - A^*\|_F^2  &\le \frac 2{nm} \|A - A^*\|_F^2  + \frac 2{nm} \|\tilde A - A\|_F^2 \nonumber \\
& \le \frac 2{nm} \|A - A^*\|_F^2 + \frac 54 \Big(\frac{\sigma^2 V(A) \log n}{n}\Big)^{2/3} + \frac{5 \sigma^2 }{4 n}\log n  \label{eq:global-triangle}
\end{align}
by noting that $\log(en) \le 2.5 \log n$ for $n \ge 2$,
and similarly
\begin{equation} \label{eq:global-rate-bound}
\frac{\sigma^2 K(\tilde A)}{nm} \log (en) \le 5 \Big(\frac{\sigma^2 V(A)\log n}{n}  \Big)^{2/3} + \frac{5 \sigma^2}{n}\log n \, .
\end{equation}
Plugging \eqref{eq:global-triangle} and \eqref{eq:global-rate-bound} into the right-hand side of \eqref{eq:oracle-loss} and \eqref{eq:oracle-mse}, and then minimizing over $A \in \mathcal U^m$, we complete the proof.

\section{Metric entropy} \label{sec:covering}

In this section, we study various \emph{covering numbers} or \emph{metric entropy} related to the parameter space of the model \eqref{eq:model}. First recall some standard definitions that date back at least to \cite{KolTih61}. An $\varepsilon$-net of a subset $G\subset {\R}^n$ with respect to a norm $\|\cdot\|$ is a set $\{w_1,\cdots, w_N\}\subset G$ such that for any $w\in G$, there exists $i\in [N]$ for which $\| w-w_i\|\leq \varepsilon$. The covering number $N(G,\|\cdot\|,\varepsilon)$ is the cardinality of the smallest $\varepsilon$-net with respect to the norm $\|\cdot\|$. Metric entropy is defined as the logarithm of a covering number. In the following, we will consider the Euclidean norm unless otherwise specified.

\subsection{Cartesian product of cones}

Lemma~\ref{lem:cover-product} below bounds covering numbers of product spaces and is useful in later proofs. We start with a well-known result on the covering number of a Euclidean ball with respect to the $\ell_\infty$-norm (see e.g. \cite[Lemma~7.14]{Mas07} for an analogous result). 

\begin{lemma} \label{lem:cube-cover-ball}
For any $\varepsilon\in(0,1]$,
\[
N\Big(\mathcal B^m(0, 1), \|\cdot\|_\infty, \frac{\varepsilon}{\sqrt m}\Big) \le (C/\varepsilon)^m, 
\]
for some constant $C>0$.
\end{lemma}

\begin{proof}
We aim at bounding the covering number of a Euclidean ball by cubes. Let $\{x^1, \dots, x^M\}$ be a maximal $\frac{\varepsilon}{\sqrt m}$-packing of $\mathcal B^m(0, 1)$ with respect to the $\ell_\infty$-norm, where a $\delta$-packing of a set $G$ with respect to a norm $\|\cdot\|$ is a set $\{w_1,\cdots,w_N\}\subset G $ such that $\|w_i-w_j\|\ge \delta$ for all distinct $i,j\in[N]$.  Then this set is necessarily an $\frac{\varepsilon}{\sqrt m}$-net of $\mathcal B^m(0, 1)$ by maximality, so $N(\mathcal B^m(0, 1), \|\cdot\|_\infty, \frac{\varepsilon}{\sqrt m}) \le M$. Consider the cubes with side length $\frac{\varepsilon}{\sqrt m}$ centered at $x^i$ for $1 \le i \le M$. These cubes are disjoint and contained in the set $\mathcal B^m(0, 1)+Q^m(\frac{\varepsilon}{\sqrt m})$, where $Q^m(\frac{\varepsilon}{\sqrt m})$ is the cube with side length $\frac{\varepsilon}{\sqrt m}$ centered at the origin in ${\R}^m$. Since $Q^m(\frac{\varepsilon}{\sqrt m}) \subset \mathcal B^m(0, \varepsilon)$,
\begin{align*}
M \operatorname{Vol}\Big(Q^m\big(\frac{\varepsilon}{\sqrt m}\big)\Big) &\le \operatorname{Vol}\Big(\mathcal B^m(0, 1)+Q^m\big(\frac{\varepsilon}{\sqrt m}\big)\Big)\\
&\le \operatorname{Vol}(\mathcal B^m(0, 1+ \varepsilon))\\
&\le \operatorname{Vol}(\mathcal B^m(0, 2)). 
\end{align*}
This proves the following bound on the covering number in terms of a volume ratio:
\[ 
N\Big(\mathcal B^m(0, 1), \|\cdot\|_\infty, \frac{\varepsilon}{\sqrt m}\Big) \le \frac{\operatorname{Vol}(\mathcal B^m(0, 2))}{\operatorname{Vol}(Q^m(\frac{\varepsilon}{\sqrt m}))} \le \frac{C^m m^{-m/2}} {\varepsilon^{m} m^{-m/2}} = (C/\varepsilon)^m.
\]
\end{proof}

Now we study the metric entropy of a Cartesian product of convex cones. 
Let $\{I_i\}_{i=1}^m$ be a partition of $[n]$ with $|I_i|=n_i$ and $\sum_{i=1}^m n_i = n$. For $a \in {\R}^n$,  the restriction of $a$ to the coordinates in $I_i$ is denoted by $a_{I_i}\in {\R}^{n_i}$. Let $\mathcal C_i$ be a convex cone in ${\R}^{n_i}$ and $\mathcal C = \mathcal C_1 \times \cdots \times \mathcal C_m$. 

\begin{lemma} \label{lem:cover-product}
With the notation above, suppose that $a_{I_i} \in \mathcal C_i \cap (-\mathcal C_i)$. Then for any $t>0$ and $\varepsilon\in (0,t]$,
\[
\log N \big(\mathcal C\cap \mathcal B^n(a,t), \|\cdot\|_2, \varepsilon \big) \le m \log \frac{Ct}{\varepsilon} + \sum_{i=1}^m \log N\Big(\mathcal C_i \cap \mathcal B^{n_i}(a_{I_i},t), \|\cdot\|_2, \frac{\varepsilon}3\Big)
\]
for some constant $C>0$.
\end{lemma}

\begin{proof}
Since a product of balls $\mathcal B^{n_1}(0,\frac{\varepsilon}{\sqrt{m}})\times\cdots\times \mathcal B^{n_m}(0,\frac{\varepsilon}{\sqrt{m}})$ is contained in $\mathcal B^{n}(0, \varepsilon)$, one could try to cover $\mathcal C\cap \mathcal B^n(a,t)$ by such products of balls. It turns out that this yields an upper bound of order  $m^{3/2}$, which is too loose for our purpose. Fortunately, the following argument corrects this dependency.

Without loss of generality, we assume that $t=1$.
We construct a $3 \varepsilon$-net of $\mathcal C \cap \mathcal B^n(a, 1)$ as follows. First, let $\mathcal N_\mathcal B$ be an  $\frac{\varepsilon}{2 \sqrt m}$-net of $\mathcal B^m(0, 1)$ with respect to the $\ell_\infty$-norm. Define 
\[
\mathcal N_\mathcal D = \Big\{\mu \in \mathcal N_\mathcal B: \min_{i\in[m]}\mu_i \ge - \frac{1}{2 \sqrt m} \Big\}. 
\]
Note that $\mu_i + \frac 1{\sqrt m} >0$ for $\mu \in \mathcal N_\mathcal D$, and let $\mathcal N_{\mu_i}$ be a $(\mu_i + \frac 1{\sqrt m})\varepsilon$-net of $\mathcal C_i \cap \mathcal B^{n_i}(a_{I_i}, \mu_i+ \frac 1{\sqrt m})$.
Define $\mathcal N_\mu = \mathcal N_{\mu_1} \times \cdots \times \mathcal N_{\mu_m}$, i.e.,
\[
\mathcal N_\mu = \{w \in {\R}^n: w = (w_{I_1}, \cdots, w_{I_m}), \ w_{I_i} \in \mathcal N_{\mu_i}\}. 
\]
We claim that $\bigcup_{\mu \in \mathcal N_\mathcal D} \mathcal N_\mu$ is an $3\varepsilon$-net of $\mathcal C \cap \mathcal B^n(a, 1)$.

Fix $v \in \mathcal C \cap \mathcal B^n(a, 1)$. Let $v_{I_i} \in {\R}^{n_i}$ be the restriction of $v$ to the component space ${\R}^{n_i}$. Then $v_{I_i} \in \mathcal C_i$. Let $\lambda \in {\R}^m$ be defined by $\lambda_i = \|v_{I_i}-a_{I_i}\|_2$, so $\|\lambda\|_2 = \|v-a\|_2 \le 1$. Hence we can find $\mu \in \mathcal N_\mathcal B$ such that $\|\mu - \lambda\|_\infty \le \frac{\varepsilon}{2\sqrt m}$. In particular, for all $i \in [m]$, $\mu_i \ge \lambda_i - \frac{\varepsilon}{2\sqrt m}\ge - \frac{1}{2 \sqrt m}$, so $\mu \in \mathcal N_\mathcal D$. Moreover, $\|v_{I_i}-a_{I_i}\|_2 = \lambda_i < \mu_i + \frac{1}{\sqrt m}$ and $v_{I_i} \in \mathcal C_i$, so by definition of $\mathcal N_{\mu_i}$, there exists $w_{I_i} \in \mathcal N_{\mu_i}$ such that $\|w_{I_i} - v_{I_i}\|_2 \le (\mu_i + \frac 1{\sqrt m})\varepsilon$. Let $w = (w_{I_1}, \dots, w_{I_m}) \in \mathcal N_\mu$. Since
\[
\sum_{i=1}^m \mu_i^2 \le \sum_{i=1}^m (\lambda_i + |\lambda_i - \mu_i|)^2 \le \sum_{i=1}^m 2 \lambda_i^2 + \frac{\varepsilon^2}{2} \le  \frac{5}{2}\, , 
\]
we conclude that 
\[ 
\|w- v\|_2^2 \le \sum_{i=1}^m \Big(\mu_i + \frac 1{\sqrt m}\Big)^2 \varepsilon^2 \le 7 \varepsilon^2. 
\]
Therefore $\bigcup_{\mu \in \mathcal N_\mathcal D} \mathcal N_\mu$ is a $3 \varepsilon$-net of $\mathcal C \cap \mathcal B^n(a, 1)$.

It remains to bound the cardinality of this net. By Lemma~\ref{lem:cube-cover-ball}, $ |\mathcal N_\mathcal D| \le |\mathcal N_\mathcal B| \le ( C/\varepsilon)^m. $
Moreover, recall that $\mathcal N_{\mu_i}$ is a $(\mu_i + \frac 1{\sqrt m})\varepsilon$-net of $\mathcal C_i \cap \mathcal B^{n_i}(a_{I_i}, \mu_i+ \frac 1{\sqrt m})$.
Since $a_{I_i} \in \mathcal C_i \cap (-\mathcal C_i)$,
for any $t>0$, $\mathcal C_i \cap \mathcal B^{n_i}(a_{I_i}, t)=\{x+a_{I_i} : x\in \mathcal C_i \cap \mathcal B^{n_i}(0, t) \}.$
Hence we can choose the net so that
\begin{align*}
|\mathcal N_{\mu_i}|&=  N\Big(\mathcal C_i \cap \mathcal B^{n_i}\big(0, \mu_i+ \frac 1{\sqrt m}\big), \|\cdot\|_2, \big(\mu_i+ \frac 1{\sqrt m}\big) \varepsilon\Big ) \\
&=  N(\mathcal C_i \cap \mathcal B^{n_i}(0, 1), \|\cdot\|_2, \varepsilon)\\
&= N(\mathcal C_i \cap \mathcal B^{n_i}(a_{I_i}, 1), \|\cdot\|_2, \varepsilon) \,.
\end{align*} 
As $|\mathcal N_{\mu}| \le \prod_{i=1}^m |\mathcal N_{\mu_i}|$, therefore
\[
\Big|\bigcup_{\mu \in \mathcal N_\mathcal D} \mathcal N_\mu\Big| 
\le \Big( \frac{C}{\varepsilon} \Big)^m \prod_{i=1}^m N(\mathcal C_i \cap \mathcal B^{n_i}(a_{I_i}, 1), \|\cdot\|_2, \varepsilon) \,.
\]
Taking the logarithm completes the proof.
\end{proof}

\subsection{Unimodal vectors and matrices}

Recall that $\mathcal S_n$ denotes the closed convex cone of increasing vectors in ${\R}^n$. First, we prove a result on the metric entropy of $\mathcal S_n$ intersecting with a ball using Lemma~\ref{lem:cover-product}. 

\begin{lemma} \label{lem:cover-vector-centered}
Let $b \in {\R}^n$ be such that  $b_1 = \cdots = b_n$. Then for any $t > 0$ and $\varepsilon > 0$,
\[ \log N(\mathcal S_n \cap \mathcal B^n(b, t), \|\cdot\|_2, \varepsilon) \le C \varepsilon^{-1} t  \log(en). \]
\end{lemma}

\begin{proof}
The majority of the proof is due to Lemma~5.1 in an old version of \cite{ChaLaf15}, but we improve their result by a factor $\sqrt{\log n}$ and provide the whole proof for completeness.

The bound holds trivially if $\varepsilon > t$, since the left-hand side is zero. It also clearly holds when $n = 1$. Hence we can assume without loss of generality that $\varepsilon \le t$ and $n = 2n' \ge 2$. Moreover, assume that $t  = 1$ for simplicity and the proof will work for any $t > 0$. Let $I=\{1, \dots, n'\}$ and observe that  
\[
\log N(\mathcal S_n \cap \mathcal B^n(b, 1), \|\cdot\|_2, \varepsilon)\leq 2\log N(\mathcal S_{n'} \cap \mathcal B^{n'}(b_{I}, 1), \|\cdot\|_2, \varepsilon/\sqrt 2 \, ) \, .
\]
Let $k$ be the smallest integer for which $2^k>n'$. We partition $I$ into $k$  blocks $A_j= I \cap [2^j, 2^{j+1})$ for $j \in [k]$ and let $m_j=|A_j|$.
Since $\mathcal S_{n'} \subset \mathcal S_{m_1}\times\cdots\times  \mathcal S_{m_k}$, Lemma~\ref{lem:cover-product} yields that
\begin{multline}\label{eq:sumcover}
\log N\big(\mathcal S_{n'} \cap \mathcal B^{n'}(b_{I}, 1), \|\cdot\|_2, \varepsilon/\sqrt{2}\, \big) \\
\leq k \log \frac{C}{\varepsilon} +\sum_{j=1}^k \log N\Big(\mathcal S_{m_j} \cap \mathcal B^{m_j}(b_{A_j},1), \|\cdot\|_2, \frac{\varepsilon}{3\sqrt{2}}\Big) \,.
\end{multline}

We know from \cite[Lemma 4.20]{Cha14} that for any $c\le d$ and $n\geq 1$,
\[
\log N\Big(\mathcal S_{n}\cap [c,d]^n \cap \mathcal B^{n}(b,1), \|\cdot\|_2, {\varepsilon}\Big)\leq \frac{C\sqrt{n}(d-c)}{\varepsilon} \,.
\] 
For each $a\in \mathcal S_{n} \cap \mathcal B^{n}(0, 1)$, it holds that $|a_i|\leq \frac{1}{\sqrt{i}}$ for $i\in I$ (since either $|a_l| \ge |a_i|$ for all $l \le i$ or $|a_l| \ge |a_i|$ for all $i \le l \le n$; see e.g. \cite{DaiRigXiaZha14}), so $\max _{i\in A_j}|a_i|\leq 2^{-j/2}$.  Also $m_j \le 2^j$, so we get that
\[
\log N\Big(\mathcal S_{m_j} \cap \mathcal B^{m_j}(b_{A_j},1), \|\cdot\|_2, \frac{\varepsilon}{3\sqrt{2}}\Big)\leq \frac{C }{\varepsilon}
\]
for all $j \in [k]$. Substituting this bound into \eqref{eq:sumcover} and noting that $k \le \log_2 n$, we reach the conclusion
\[
\log N(\mathcal S_{n'} \cap \mathcal B^{n'}(b_{I}, 1), \|\cdot\|_2, \varepsilon/\sqrt{2}) \le C \varepsilon^{-1} \log (en) \,.
\]
\end{proof}


Next, we study the metric entropy of the set of matrices with unimodal columns.
Recall that $\mathcal C_l=\{a \in {\R}^n: a_1 \le \cdots \le a_l\}\cap\{a \in {\R}^n: a_l\geq \cdots \geq a_n\}$ for $l \in [n]$. 
For $\mathbf l = (l_1, \dots, l_m) \in [n]^m$,  define $\mathcal C_\mathbf{l}^m=\mathcal C_{l_1}\times \cdots\times \mathcal C_{l_m}$. 
Moreover, for $A \in \R^{n \times m}$, $t>0$ and $\mathcal C \subset \R^{n \times m}$, define 
\begin{align}
\Theta_{\mathcal C}(A, t) &= \bigcup_{\lambda \ge 0} \{B - \lambda A: B \in \mathcal C \cap \mathcal B^{nm}(\lambda A, t) \} \label{eq:def-theta} \\
&= \bigcup_{\lambda \ge 0} \Big(\mathcal C \cap \mathcal B^{nm}(\lambda A, t) - \lambda A \Big) . \nonumber
\end{align}
Note that in particular $\Theta_{\mathcal C}(A, t) \subset \mathcal B^{nm}(0, t)$.

\begin{lemma} \label{lem:cover-matrix-uni}
Given $A \in {\R}^{n\times m}$ and $\mathbf l = (l_1, \dots, l_m) \in [n]^m$, define $k(A_{\cdot,j}) = \card(\{A_{1,j}, \dots, A_{n,j}\})$ and $K(A) = \sum_{j=1}^m k(A_{\cdot,j}).$ Then for any $t > 0$ and $\varepsilon > 0$,
\[ \log N \big( \Theta_{\mathcal C_\mathbf{l}^m}(A, t), \|\cdot\|_F, \varepsilon \big) \le C \varepsilon^{-1} t \, K(A) \log \frac{e n m}{K(A)}. \]
\end{lemma}

\begin{proof}
Assume that $\varepsilon \le t$ since otherwise the left-hand side is zero and the bound holds trivially. For $j \in [m]$, define $I^{j, 1} = [l_j]$ and $I^{j, 2} = [n] \setminus [l_j]$. Define $k_{j,1} = k(A_{I^{j,1},j})$ and $k_{j,2} = k(A_{I^{j, 2},j})$. Let $\varkappa = \sum_{j=1}^m (k_{j,1} + k_{j,2})$ and observe that $K(A)\le  \varkappa \le 2 K(A).$
Moreover, let $\{I^{j,1}_1, \dots, I^{j,1}_{k_{j,1}}\}$ be the partition of $I^{j,1}$ such that $A_{I^{j,1}_i, j}$ is a constant vector for $i \in [k_{j,1}]$. Note that elements of $I^{j,1}_i$ need not to be consecutive. Define the partition for $I^{j,2}$ analogously.

For $j \in [m]$ and $i \in [k_{j,1}]$ (resp. $[k_{j,2}]$), let $\mathcal S_{I^{j, 1}_i,j}$ (resp. $\mathcal S_{I^{j, 2}_i,j}$) denote the set of increasing (resp. decreasing) vectors in the component space ${\R}^{|I^{j, 1}_i|}$ (resp. ${\R}^{|I^{j, 2}_i|}$). 
Lemma~\ref{lem:cover-vector-centered} implies that
\[ \
\log N(\mathcal S_{I^{j, r}_i,j} \cap \mathcal B^{|I^{j, r}_i|}(A_{I^{j, r}_i, j}, t), \|\cdot\|_F, \varepsilon) \le C \varepsilon^{-1} t \log (e|I^{j, r}_i|).
\]
As a matrix in $\R^{n \times m}$ can be viewed as a concatenation of $\varkappa = \sum_{j=1}^m (k_{j,1}+k_{j,2})$ vectors of length $|I^{j, r}_i|, r \in [2], j \in [m]$, we define the cone $\mathcal S^*$ in $\R^{n\times m}$ by
$\mathcal S^* = \prod_{j=1}^m \prod_{r=1}^2 \prod_{i=1}^{k_{j, r}} \mathcal S_{I^{j, r}_i,j}$, which is clearly a superset of $\mathcal C_{\mathbf l}^m$.  It also  follows that $A \in \mathcal S^* \cap (-\mathcal S^*)$,
and thus by Lemma~\ref{lem:cover-product} and the previous display,
\begin{align*} 
\log N(\mathcal S^* \cap \mathcal B^{nm}(A, t), \|\cdot\|_F, \varepsilon)
&\le \varkappa \log \frac {C t}\varepsilon + \sum_{j=1}^m \sum_{r=1}^2 \sum_{i=1}^{k_{j, r}} C \varepsilon^{-1} t \log (e |I^{j, r}_i|)  \\
& \le C \varepsilon^{-1} t \, \varkappa + C \varepsilon^{-1} t \, \varkappa \log \frac{e \sum_{j,r,i} |I^{j, r}_i|}{\varkappa} \\
&\le C  \varepsilon^{-1} t \, K(A) \log \frac{enm}{K(A)},
\end{align*}
where we used the  concavity of the logarithm and Jensen's inequality in the second step, and that $K(A) \le \varkappa \le 2K(A)$ in the last step.

Since $A \in \mathcal S^* \cap (-\mathcal S^*)$ (the cone $\cS^*$ is pointed at $A$) we have that $\mathcal S^* \cap \mathcal B^{nm}(\lambda A, t) - \lambda A = \mathcal S^* \cap \mathcal B^{nm}(0, t)$ for any $\lambda \ge 0$. In view of Definition \eqref{eq:def-theta}, it holds
$$
\Theta_{\mathcal S^*}(A, t)=\bigcup_{\lambda\ge 0} \mathcal S^* \cap \mathcal B^{nm}(\lambda A, t) - \lambda A= \mathcal S^* \cap \mathcal B^{nm}(\lambda A, t) - \lambda A\,,\quad \forall \lambda\ge 0\,.
$$ 
In particular, taking $\lambda=1$, we get $\Theta_{\mathcal S^*}(A, t)=\mathcal S^* \cap \mathcal B^{nm}(A, t) - A$.
Moreover, $\mathcal C_{\mathbf l}^m \subset \mathcal S^*$, so that $\Theta_{\mathcal C_{\mathbf l}^m}(A, t) \subset \Theta_{\mathcal S^*}(A, t) = \mathcal S^* \cap \mathcal B^{nm}(A, t) - A$. Thus the metric entropy of $\Theta_{\mathcal C_{\mathbf l}^m}(A, t)$ is subject to the above bound as well.
\end{proof}

Finally, we consider the metric entropy of $\Theta_{\mathcal M}(A, t)$ for $A \in \R^{n \times m}$, $t>0$ and $\mathcal M = \bigcup_{\Pi \in \mathfrak S_n} \Pi \mathcal U^m$. The above analysis culminates in the following lemma which we use to prove the main upper bounds.

\begin{lemma} \label{lem:cover-matrix-union}
Let $A \in \R^{n \times m}$ and $K(A)$ be defined as in the previous lemma. Then for any $\varepsilon > 0$ and $t>0$,
\[ \log N \big(\Theta_{\mathcal M}(A, t), \|\cdot\|_F, \varepsilon \big) \le C \varepsilon^{-1} t\, K(A)  \log \frac{e n m}{K(A)} +  n  \log n. \]
\end{lemma}

\begin{proof}
Assume that $\varepsilon \le t$ since otherwise the left-hand side is zero and the bound holds trivially.
Note that $\mathcal U^m = \bigcup_{\mathbf l \in [n]^m}  \mathcal C_\mathbf l^m$, and that $\mathcal M = \bigcup_{\Pi \in \mathfrak S_n} \Pi \mathcal U^m$. Thus $\mathcal M$ is the union of $n^m n!$ cones of the form $\Pi \mathcal C_\mathbf l^m$. By Definition \eqref{eq:def-theta}, $\Theta_\mathcal M(A, t)$ is also the union of $n^m n!$ sets $\Theta_{\Pi \mathcal C_\mathbf l^m} (A, t)$, each having metric entropy subject to the bound in Lemma~\ref{lem:cover-matrix-uni}. Therefore, a union bound implies that
\begin{align*} 
\log N \big(\Theta_{\mathcal M}(A, t), \|\cdot\|_F, \varepsilon \big) 
&\le \log  N \big(\Theta_{\mathcal C_\mathbf l^m} (A, t), \|\cdot\|_F, \varepsilon \big) + \log(n^m n!) \\
&\le C \varepsilon^{-1} t \, K(A)  \log \frac{e n m}{K(A)} + m \log n +n \log n\\
& \le C \varepsilon^{-1} t \, K(A)  \log \frac{e n m}{K(A)} + n \log n,
\end{align*}
where the last step follows from that $K \log(enm/K) \ge m \log n$ for $m \le K \le nm$ and that $\varepsilon \le t$.
\end{proof}

\section{Proof of the lower bounds} \label{sec:lower}

For minimax lower bounds, we consider the model $Y = \Pi^* A^* + Z$ where entries of $Z$ are i.i.d. $N(0, \sigma^2)$. 
The Varshamov-Gilbert lemma \cite[Lemma~4.7]{Mas07} is a standard tool for proving lower bounds.

\begin{lemma}[Varshamov-Gilbert] \label{lem:vg}
Let $\delta$ denote the Hamming distance on $\{0, 1\}^d$ where $d \ge 2$. Then there exists a subset $\Omega \subset \{0, 1\}^d$ such that $\log |\Omega| \ge d/8$ and $\delta(\omega, \omega') \ge d/4$ for distinct $\omega, \omega' \in \Omega$.
\end{lemma}

We also need the following useful lemma.

\begin{lemma} \label{lem:lower-general}
Consider the model $y =\theta + z$ where  $\theta \in \Theta \subset {\R}^d$ and $z \sim N(0, \sigma^2 I_d)$.
Suppose that $|\Theta| \ge 3$ and for distinct $\theta, \theta' \in \Theta$,
$4\phi \le \|\theta - \theta'\|_2^2 \le \frac{\sigma^2}{8} \log |\Theta|$ where $\phi > 0$.
Then there exists $c>0$ such that
\[ \inf_{\hat \theta} \sup_{\theta \in \Theta} {\p}_\theta \big[ \|\hat \theta - \theta\|_F^2 \ge \phi \big] \ge c. \]
\end{lemma}

\begin{proof}
Let ${\p}_\theta$ denote the probability with respect to $\theta + z$. Then the Kullback-Leibler divergence between ${\p}_{\theta}$ and ${\p}_{\theta'}$ satisfies that
\[ \operatorname{KL}({\p}_{\theta}, {\p}_{\theta'}) =  \frac{\|\theta - \theta'\|_F^2}{2\sigma^2} \le \frac{\log |\Theta|}{16} \le \frac{\log (|\Theta|-1)}{10} \,, \]
since $|\Theta| \ge 3$. Applying \cite[Theorem~2.5]{Tsy09} with $\alpha = \frac 1{10}$ gives the conclusion.
\end{proof}

\subsection{Proof of Theorem~\ref{thm:lower-adaptive}}

We define $\mathcal U_{K_0}^m(V_0) = \mathcal U_{K_0}^m \cap \mathcal U^m(V_0)$ and $\mathcal M_{K_0}(V_0) = \bigcup_{\Pi \in \mathfrak S_n} \Pi \, \mathcal U_{K_0}^m(V_0)$. Define the subset of $\mathcal M_{K_0}(V_0)$ containing permutations of monotonic matrices by $\mathcal M_{K_0}^{\mathcal S}(V_0) = \{ \Pi A \in \mathcal M_{K_0}(V_0) : \Pi \in \mathfrak S_n, A \in \mathcal S^m\}$.
Since each estimator pair $(\hat \Pi, \hat A)$ gives an estimator $\hat M = \hat \Pi \hat A$ of $M = \Pi A$, it suffices to prove a lower bound on $\|\hat M - M\|_F^2$. 
In fact, we prove a stronger lower bound than the one in Theorem~\ref{thm:lower-adaptive}.

\begin{proposition} \label{prop:lower-stronger}
Suppose that $K_0 \le m (\frac{16n}{\sigma^2})^{1/3} V_0^{2/3} - m$. Then
\begin{multline} \label{eq:lower-stronger}
\inf_{\hat M} \sup_{M \in \mathcal M_{K_0}(V_0)} {\p}_{M} \Big[ \frac 1{nm} \|\hat M - M\|_F^2 \ge c \sigma^2 \frac{K_0}{nm} \\
+ c \max_{1 \le l \le \min(K_0-m, m)+1} \min\Big(\frac {\sigma^2}m \log l,  m^2 l^{-3} V_0^2\Big) \Big] \ge c'
\end{multline}
for some $c, c'>0$, where ${\p}_{M}$ is the probability with respect to $Y= M+ Z$. This bound remains valid for the parameter subset $\mathcal M_{K_0}^{\mathcal S}(V_0)$ if $l=1$ or $2$.
\end{proposition}

Note that the bound clearly holds for the larger parameter space $\mathcal M_{K_0} = \bigcup_{\Pi \in \mathfrak S_n} \Pi \mathcal U_{K_0}^m$. By taking $l = \min(K_0-m, m) + 1$ and $V_0$ large enough, we see that the assumption in Proposition~\ref{prop:lower-stronger} is satisfied and the second term becomes simply $\frac{\sigma^2}m \log l$, so Theorem~\ref{thm:lower-adaptive} follows.
In the monotonic case, by the last statement of the proposition, if $K_0 \ge m+1$ then taking $l=2$ and $V_0$ large enough yields a lower bound of rate $\sigma^2 (\frac{K_0}{nm}
+ \frac {1}m)$ for the set of matrices $A$ with increasing columns and $K(A) \le K_0$.

The proof of Proposition~\ref{prop:lower-stronger} has two parts which correspond to the two terms respectively. First, the term $\sigma^2 \frac{K_0}{nm}$ is derived from the proof of lower bounds for isotonic regression in \cite{BelTsy15}.
Then we derive the other term $\frac{\sigma^2}{m} \log l$ for any $1 \le l \le \min(K_0 - m, m)+1$, which is due to the unknown permutation.

\begin{lemma} \label{lem:lower-est}
Suppose that $K_0 \le m (\frac{16n}{\sigma^2})^{1/3} V_0^{2/3} - m$. For some $c, c'>0$,
\[ \inf_{\hat M} \sup_{M \in \mathcal M^{\mathcal S}_{K_0}(V_0)} {\p}_M\big[ \|\hat M - M\|_F^2 \ge c \sigma^2 K_0 \big] \ge c\,, \]
where ${\p}_M$ is the probability with respect to $Y= M+ Z$.
\end{lemma}

\begin{proof}
We adapt the proof of \cite[Theorem~4]{BelTsy15} to the case of matrices. Let $V_j = V_0$ for all $j \in [m]$. Since 
\[ K_0 \le m \big(\frac{16n}{\sigma^2}\big)^{1/3} V_0^{2/3} - m 
= \sum_{j=1}^m \Big[\big(\frac{16n}{\sigma^2}\big)^{1/3}V_j^{2/3} - 1 \Big],
\]
we can choose $k_j \in [n]$ so that $k_j \le (\frac{16n}{\sigma^2})^{1/3}V_j^{2/3}$ and $K_0 = \sum_{j=1}^m k_j$. 
According to Lemma~\ref{lem:vg}, there exists $\Omega \subset \{0,1\}^{K_0}$ such that $\log |\Omega| \ge K_0/8$ and $\delta(\omega, \omega') \ge K_0/4$ for distinct $\omega, \omega' \in \Omega$. Consider the partition $[K_0] = \cup_{m=1}^j I_j$ with $|I_j| = k_j$. For each $\omega \in \Omega$, let $\omega^j \in \{0,1\}^{k_j}$ be the restriction of $\omega$ to coordinates in $I_j$. Define $M^\omega \in {\R}^{n \times m}$ by
\[ M_{i,j}^{\omega} = \frac{\lfloor (i-1) k_j/ n \rfloor V_j}{2 k_j} + \gamma_j \omega_{\lfloor (i-1) k_j/ n \rfloor + 1}, \]
where $\gamma_j = \frac \sigma 8 \sqrt{k_j / 2n}$. It is straightforward to check that $k(M_{\cdot,j}) \le k_j$, $V(M_{\cdot,j}) \le V_j$ and $M_{\cdot, j}$ is increasing, so 
$M$ is in the parameter space.
Moreover, for distinct $\omega, \omega' \in \Omega$,
\[ \|M^\omega - M^{\omega'}\|_F^2 \ge c \sum_{j=1}^m \frac{n}{k_j} \gamma_j^2 \delta(\omega^j, (\omega')^j) \ge c \sigma^2 \sum_{j=1}^m \delta(\omega^j, (\omega')^j) = c\sigma^2 K_0. \]
On the other hand,
\[ \|M^\omega - M^{\omega'}\|_F^2 \le 2 \sum_{j=1}^m \frac{n}{k_j} \gamma_j^2 \delta(\omega^j, (\omega')^j) \le \frac{\sigma^2}{64} \delta(\omega, \omega') \le \frac{\sigma^2 K_0}{64} \le \frac{\sigma^2}{8} \log |\Omega|. \]
Applying Lemma~\ref{lem:lower-general} completes the proof.
\end{proof}

For the second term in \eqref{eq:lower-stronger}, we first note that the bound is trivial for $l =1$ since $\log l = 0$. The next lemma deals with the case $l=2$.

\begin{lemma} \label{lem:lower-monotone}
There exist constants $c, c'>0$ such that for any $K_0 \ge m+1$ and $V_0 \ge 0$, 
\[
\inf_{\hat M} \sup_{M \in \mathcal M^{\mathcal S}_{K_0}(V_0)} {\p}_M\Big[ \|\hat M - M\|_F^2 \ge c n \min\big(\sigma^2,  m^3 V_0^2\big) \Big] \ge c' \,,
\]
where ${\p}_M$ is the probability with respect to $Y= M+ Z$.
\end{lemma}

\begin{proof}
By Lemma~\ref{lem:vg}, there exists $\Omega \subset \{0, 1\}^n$ such that $\log |\Omega| \ge n/8$ and $\delta(\omega, \omega') \ge n/4$ for distinct $\omega, \omega' \in \Omega$. For each $\omega \in \Omega$, define $M^\omega \in {\R}^{n\times m}$ by setting the first column of $M^\omega$ to be $\alpha \omega$ and all other entries to be zero, where
$\alpha = \min\big(\frac \sigma 8,  m^{3/2} V_0\big).$
Then 
\begin{enumerate}
\item $M^\omega \in \mathcal M^{\mathcal S}_{K_0}(V_0)$ since $K(M) = m+1 \le K_0$, $V(M) \le V_0$ and we can permutate the rows of $M^\omega$ so that its first column is increasing;

\item $\|M^\omega - M^{\omega'}\|_F^2 \ge \min(\frac{\sigma^2}{64}, m^3 V_0^2)\, \delta(\omega, \omega') \ge \min(\frac{n \sigma^2}{256},  \frac n4 m^3 V_0^2)$ for distinct $\omega, \omega' \in \Omega$;

\item $\|M^\omega - M^{\omega'}\|_F^2 \le \frac{\sigma^2}{64} \delta(\omega, \omega') \le \frac{\sigma^2}{64} n \le \frac{\sigma^2} 8 \log |\Omega|$ for $\omega, \omega' \in \Omega$. 
\end{enumerate}
Applying Lemma~\ref{lem:lower-general} completes the proof.
\end{proof}


For the previous two lemmas, we have only used matrices with increasing columns. However, to achieve the second term in \eqref{eq:lower-stronger} for $l \ge 3$, we need matrices with unimodal columns. 
The following packing lemma is the key.

\begin{lemma} \label{lem:zero-one-packing}
For $l \in [m]$, consider the set $\mathfrak M$ of $n \times m$ matrices of the form
\[ M = 
\begin{cases}
1 & \text{for exactly one } j_i \in [l] \text{ for each } i \in [n], \\
0 & \text{otherwise} .
\end{cases}
\]
For $\varepsilon> 0$, define $k = \lfloor \frac{\varepsilon^2 n}2 \rfloor$.
Then there exists an $\varepsilon \sqrt n$-packing $\mathcal P$ of $\mathfrak M$ such that $|\mathcal P| \ge l^{n-k} (\frac k{en})^k$ if $k \ge 1$ and $|\mathcal P| = l^n$ if $k = 0$.
\end{lemma}

\begin{proof}
There are $l$ choices of entries to put the one in each row of $M$, so $|\mathfrak M| = l^n$. Fix $M_0 \in \mathfrak M$. 
If $\|M - M_0\|_F \le \varepsilon \sqrt n$ where $M \in \mathfrak M$, then $M$ differs from $M_0$ in at most $k$ rows. If $k = 0$, taking $\mathcal P = \mathfrak M$ gives the result. If $k \ge 1$ then 
\[ \big|\mathfrak M \cap B^{nm}(M_0, \varepsilon \sqrt n)\big| \le \binom{n}{k} l^k \le \big(\frac{en}{k}\big)^k l^k \,. \]
Moreover, let $\mathcal P$ be a maximal $\varepsilon\sqrt n$-packing of $\mathfrak M$. Then $\mathcal P$ is also an $\varepsilon \sqrt n$-net, so 
$\mathfrak M \subset \bigcup_{M_0 \in \mathcal P} B^{nm} (M_0, \varepsilon \sqrt n).$
It follows that
\[ l^n = |\mathfrak M| \le \sum_{M_0 \in \mathcal P} \big|\mathfrak M \cap B^{nm} (M_0, \varepsilon \sqrt n)\big| \le |\mathcal P| \cdot \big(\frac{en}{k}\big)^k l^k \,. \]
We conclude that $|\mathcal P| \ge l^{n-k} (\frac k{en})^k$.
\end{proof}

For notational simplicity, we now consider $2 \le l  \le \min(K_0 - m, m)$ instead of $3 \le l  \le \min(K_0 - m, m)+1$.

\begin{lemma} \label{lem:lower-approx}
There exist constants $c, c'>0$ such that for any $K_0 \ge m$, $V_0 \ge 0$ and $2 \le l  \le \min(K_0 - m, m)$, 
\[
\inf_{\hat M} \sup_{M \in \mathcal M_{K_0}(V_0)} {\p}_M\Big[ \|\hat M - M\|_F^2 \ge c n \min\big(\sigma^2 \log (l+1),  m^3 (l+1)^{-3} V_0^2\big) \Big] \ge c' \,,
\]
where ${\p}_M$ is  the probability with respect to $Y= M+ Z$.
\end{lemma}

\begin{proof}
Set $\varepsilon = 1/2$ and let $\mathcal P$ be the $\sqrt n /2$-packing  given by Lemma~\ref{lem:zero-one-packing}. If $k = \lfloor \frac n8 \rfloor = 0$, then $\log |\mathcal P| = n \log l$. Now assume that $k \ge 1$. Since $(\frac{x}{en})^x$ is decreasing on $[1, n]$, we have that
$|\mathcal P| \ge l^{7n/8} (\frac 1{8e})^{n/8}$. Hence for $l \ge 2$,
\begin{equation} \label{eq:log-p} 
\log |\mathcal P| \ge \frac{7n}8 \log l - \frac n8 \log(8e) \ge \frac n4 \log l \, . 
\end{equation}

Moreover, for each $M_0 \in \mathcal P$, consider the rescaled matrix 
\[ M = \min\Big(\frac{\sigma}{8}\sqrt{\frac{\log l}2}, \big(\frac ml\big)^{3/2} V_0\Big) M_0 \,.\]
\begin{enumerate}
\item We can permute the rows of $M_0$ so that each column has consecutive ones (or all zeros), so $M \in \mathcal M$. Moreover, 
\[ K(M) = 2l+m-l \le \min(m, K_0 - m) + m \le K_0 \]
and 
\[ V(M) \le \Big( \frac 1m \sum_{j = 1}^l \big( (m/l)^{3/2} V_0 \big)^{2/3} \Big)^{3/2}  = V_0 \,, \]
so
$M \in \mathcal M_{K_0}(V_0)$ for $M_0 \in \mathcal P$.

\item For $M_0, M_0' \in \mathcal P$, $\|M_0 - M_0'\|_F^2 \ge n/4$, so
\begin{align*}
\|M - M'\|_F^2 &= \min\Big(\frac{\sigma^2 \log l}{128}, (m/l)^3 V_0^2\Big) \|M_0 - M_0'\|_F^2 \\ 
&\ge \min \Big( \frac{\sigma^2}{512} n \log l, \frac n4 \big(\frac ml \big)^3 V_0^2 \Big) \,. 
\end{align*}

\item For $M_0, M_0' \in \mathcal P$, $\|M_0 - M_0'\|_F^2 \le 2\|M_0\|_F^2 + 2\|M_0'\|_F^2 \le 4n$, so by \eqref{eq:log-p},
\[\|M - M'\|_F^2 \le \frac{\sigma^2 \log l}{128}\|M_0 - M_0'\|_F^2
\le \frac{\sigma^2}{32} n \log l \le \frac{\sigma^2} 8 \log |\mathcal P| \,. \]
\end{enumerate}
Since $\log l \ge \frac 12\log(l+1)$ for $l \ge 2$, applying Lemma~\ref{lem:lower-general} completes the proof.
\end{proof}

Combining Lemma~\ref{lem:lower-est}, \ref{lem:lower-monotone} and \ref{lem:lower-approx}, and dividing the bound by $nm$, 
we get \eqref{eq:lower-stronger}
because the max of two terms is lower bounded by a half of their sum. The last statement in Proposition~\ref{prop:lower-stronger} holds since Lemma~\ref{lem:lower-est} and \ref{lem:lower-monotone} are proved for matrices with increasing columns.

\subsection{Proof of Theorem~\ref{thm:lower-global}}

The proof will only use Lemma~\ref{lem:lower-est} and \ref{lem:lower-monotone}, so the lower bound of rate $(\frac{\sigma^2 V_0}{n})^{2/3} +  \frac {\sigma^2}n
+ \min (\frac {\sigma^2}m,  m^2 V_0^2)$ 
holds even if the matrices are required to have increasing columns.

The last term $\min (\frac {\sigma^2}m,  m^2 V_0^2)$ is achieved by Lemma~\ref{lem:lower-monotone}, so we focus on the trade-off between the first two terms.
Suppose that $(\frac{16n}{\sigma^2})^{1/3} V_0^{2/3} \ge 3$, in which case the first term $(\frac{\sigma^2 V_0}{n})^{2/3}$ dominates the second term. Then $m (\frac{16n}{\sigma^2})^{1/3} V_0^{2/3} - m \ge 2m$. Setting
\[ K_0 = \big\lfloor m \big(\frac{16n}{\sigma^2}\big)^{1/3} V_0^{2/3} - m \big\rfloor  \,,\]
we see that 
$ K_0 \ge \big\lfloor \frac m 2 (\frac{16n}{\sigma^2})^{1/3} V_0^{2/3}\big\rfloor$. Lemma~\ref{lem:lower-est} can be applied with this choice of $K_0$. Then the term $c \sigma^2 \frac{K_0}{nm}$ is lower bounded by $c(\frac{\sigma^2 V_0}{n})^{2/3}$.

On the other hand, if $(\frac{16n}{\sigma^2})^{1/3} V_0^{2/3} \le 3$, then the second term $\frac{\sigma^2}{n}$ dominates the first up to a constant. To deduce a lower bound of this rate, we apply Lemma~\ref{lem:vg} to get $\Omega \subset \{0, 1\}^m$ such that $\log |\Omega| \ge m/8$ and $\delta(\omega, \omega') \ge m/4$ for distinct $\omega, \omega' \in \Omega$. For each $\omega \in \Omega$, define $M^\omega \in {\R}^{n\times m}$ by setting every row of $M^\omega$ equal to $\frac{\sigma}{8\sqrt n} \omega^\top$. Then 
\begin{enumerate}
\item $M^\omega \in \mathcal U^m(V_0)$ since $V(M^\omega) = 0$;
\item $\|M^\omega - M^{\omega'}\|_F^2 = \frac{\sigma^2}{64} \delta(\omega, \omega') \ge c \sigma^2 m$;
\item $\|M^\omega - M^{\omega'}\|_F^2 = \frac{\sigma^2}{64} \delta(\omega, \omega') \le \frac{\sigma^2}{64} m \le \frac{\sigma^2}8 \log |\Omega|$.
\end{enumerate}
Hence Lemma~\ref{lem:lower-general} implies a lower bound on $\frac{1}{nm} \|\hat M - M\|_F^2$ of rate $\frac{\sigma^2 m}{nm} = \frac{\sigma^2}{n}$.

\section{Matrices with increasing columns} \label{sec:monotone}

For the model $Y = \Pi^* A^* + Z$ where $A^* \in \mathcal S^m$ and $Z \sim \operatorname{subG}(\sigma^2)$, a computationally efficient estimator $(\tilde \Pi, \tilde A)$ has been constructed in Section~\ref{sec:rankscore} using the \RankScore\ procedure. We will bound its rate of estimation in this section. Recall that the definition of $(\tilde \Pi, \tilde A)$ consists of two steps. First, we recover an order (or a ranking) of the rows of $Y$, which leads to an estimator $\tilde \Pi$ of the permutation. Then define $\tilde A \in \mathcal S^m$ so that $\tilde \Pi \tilde A$ is the projection of $Y$ onto the convex cone $\tilde \Pi \mathcal S^m$. For the analysis of the algorithm, we  deal with the projection step first, and then turn to learning the permutation.

\subsection{Projection}

In fact, for \emph{any} estimator $\tilde \Pi$, if $\tilde A$ is defined as above by the projection corresponding to $\tilde \Pi$, then the error $\|\tilde \Pi \tilde A - \Pi^* A^*\|_F^2$ can be split
into two parts: the permutation error $\|(\tilde\Pi - \Pi^*) A^*\|_F^2$ and the estimation error of order $\tilde O(\sigma^2 K(A^*))$. 

The proof of the following oracle inequality is very similar to that of Theorem~\ref{thm:adaptive},
so we will sketch the proof without providing all the details.

\begin{lemma} \label{lem:approx-est}
Consider the model $Y = \Pi^* A^* + Z$ where $A^* \in \mathcal S^{m}$ and $Z\sim \operatorname{subG}(\sigma^2)$. For any $\tilde \Pi \in \mathfrak S_n$, define $\tilde A \in \mathcal S^m$ so that $\tilde \Pi \tilde A$ is the projection of $Y$ onto $\tilde \Pi \mathcal S^m$. 
Then with probability at least $1-e^{-c(n+m)}$,
\begin{multline*} 
\|\tilde \Pi \tilde A - \Pi^* A^*\|_F^2 \lesssim \min_{A \in \mathcal S^m}  \Big(\|A - A^*\|_F^2 + \sigma^2 K(A) \log \frac{enm}{K(A)} \Big) + \|(\tilde\Pi - \Pi^*) A^*\|_F^2.
\end{multline*}
\end{lemma}

\begin{proof}
Assume without loss of generality that $\Pi^* = I_n$. 
Let $A \in \mathcal S^m$ and define
\[ f_{\tilde \Pi A}(t) = \sup_{M \in \tilde \Pi \mathcal S^m \cap \mathcal B^{nm}(\tilde \Pi A, t)} \langle M - \tilde \Pi A, Y - \tilde \Pi A\rangle - \frac {t^2}2 \,. \]
Since $\mathcal S^m = \mathcal C_{\mathbf l}^m$ with $\mathbf l = (n, \dots, n)$, by Lemma~\ref{lem:cover-matrix-uni},
\[ \log N\big( \Theta_{\tilde \Pi \mathcal S^m} (A, t), \|\cdot \|_F, \varepsilon \big) \le C \varepsilon^{-1} t \, K(A) \log \frac{enm}{K(A)} \,.\] 
Using the proof of Lemma~\ref{lem:f-bound}, we see that
\[ f_{\tilde \Pi A}(t) \le C \sigma t \sqrt{K(A) \log \frac{enm}{K(A)}} + t \|\tilde \Pi A - A^*\|_F - \frac{t^2}2 + st \]
with probability at least $1-C \exp(-\frac{cs^2}{\sigma^2 })$.
Then the proof of Theorem~\ref{thm:adaptive} implies that with probability at least $1-e^{-c(n+m)}$,
\begin{align*} 
\|\tilde \Pi \tilde A - A^*\|_F^2 &\lesssim \sigma^2 K(A) \log \frac{enm}{K(A)} +  \|\tilde \Pi A - A^*\|_F^2 \\
& \lesssim \sigma^2 K(A) \log \frac{enm}{K(A)} +  \|A - A^*\|_F^2 +  \|\tilde \Pi A^* - A^*\|_F^2  \, .
\end{align*}
Minimizing over $A \in \mathcal S^m$ yields the desired result.
\end{proof}

The idea of splitting the error into two terms as in Lemma~\ref{lem:approx-est} has appeared in \cite{ShaBalGunWai15, ChaMuk16}.

\subsection{Permutation} \label{sec:permutation}

By virtue of Lemma~\ref{lem:approx-est}, it remains to control the permutation error $\|\tilde \Pi A^* -\Pi^* A^*\|_F^2$ where $\tilde \Pi$ is given by the \RankScore\ procedure defined in Section~\ref{sec:rankscore}.
Recall that for $i, i' \in [n]$,
\[
\Delta_{A^*}(i, i') = \max_{j \in [m]} (A^*_{i', j} - A^*_{i, j})\vee \frac 1{\sqrt m} \sum_{j=1}^m (A^*_{i', j} - A^*_{i, j}) 
\]
and $\Delta_Y(i, i')$ is defined analogously. 
Since columns of $A^*$ are increasing,
\begin{equation} \label{eq:def-del-app}
\big|\Delta_{A^*}(i, i') \big| = \|A^*_{i', \cdot} - A^*_{i, \cdot}\|_\infty\vee \frac 1{\sqrt m} \|A^*_{i', \cdot} - A^*_{i, \cdot}\|_1  \,. 
\end{equation}
Recall that the \RankScore\ procedure is defined as follows.
First, for $i \in [n]$, we associate with the $i$-th row of $Y$ a score $s_i$ defined by
\begin{equation} \label{eq:score-app}
s_i =  \sum_{l=1}^n \1(\Delta_Y(l, i) \ge 2\tau) 
\end{equation}
for the threshold $\tau := 3 \sigma \sqrt{\log(nm \delta^{-1})}$ where $\delta$ is the probability of failure. Then we order the rows of $Y$ so that the scores are increasing with ties broken arbitrarily. 

This is equivalent to requiring that the corresponding permutation $\tilde \pi:[n] \to [n]$ satisfies that if $s_i < s_{i'}$ then $\tilde \pi^{-1}(i) < \tilde \pi^{-1}(i')$. Define $\tilde \Pi$ to be the $n \times n$ permutation matrix corresponding to $\tilde \pi$ so that $\tilde \Pi_{\tilde \pi(i), i} = 1$ for $i \in [n]$ and all other entries of $\tilde \Pi$ are zero. Moreover, let $\pi^*:[n] \to [n]$ be the permutation corresponding to $\Pi^*$.

To control the permutation error, we first state a lemma which asserts that if the gap between two rows of $A^*$ is sufficiently large, then the permutation defined above will recover their relative order with high probability.

\begin{lemma} \label{lem:row-gap}
There is an event $\mathcal E$ of probability at least $1- \delta$ on which the following holds. For any $i, i' \in [n]$, if $\Delta_{A^*}(i, i') \ge 4 \tau$, then $\tilde \pi^{-1} \circ \pi^*(i) < \tilde \pi^{-1} \circ \pi^*(i')$.
\end{lemma}

\begin{proof}
Since $Z \sim \operatorname{subG}(\sigma^2)$, $Z_{i,j}$ and $\frac 1{\sqrt m} \sum_{j=1}^m  Z_{i, j}$ are sub-Gaussian random variables with variance proxy $\sigma^2$. A standard union bound yields that
\[ \max\Big(\max_{i \in [n], j \in [m]} |Z_{i,j}|, \max_{i \in [n]} \frac 1{\sqrt m} \Big|\sum_{j=1}^m  Z_{i, j}\Big|\Big) \le \tau = 3 \sigma \sqrt{\log(nm \delta^{-1})} \]
on an event $\mathcal E$ of probability at least $1-2(nm+n)\exp(-\frac{\tau^2}{2\sigma^2}) \ge 1-\delta$.

In the sequel, we make statements that are valid on the event $\mathcal E$. Since $Y_{\pi^*(i), j} = A^*_{i,j} + Z_{i, j}$, by the triangle inequality,
\begin{equation} \label{eq:error-gap}
|\Delta_Y(\pi^*(i), \pi^*(i')) - \Delta_{A^*}(i, i')| \le 2 \tau. 
\end{equation}
Suppose that $\Delta_{A^*}(i, i') \ge 4 \tau$. We claim that $s_{\pi^*(i)} < s_{\pi^*(i')}$.  If for $l \in [n]$, $\Delta_Y(\pi^*(l), \pi^*(i))\ge 2 \tau$, then $\Delta_{A^*}(l, i) \ge 0$ by \eqref{eq:error-gap}. Since $A^*$ has increasing columns, $\Delta_{A^*}(l, i') \ge 4 \tau$. Again by \eqref{eq:error-gap}, $\Delta_Y(\pi^*(l), \pi^*(i')) \ge 2 \tau$. By definition \eqref{eq:score-app}, we see that $s_{\pi^*(i)} \le s_{\pi^*(i')}$. Moreover, $\Delta_{A^*}(i, i') \ge 4 \tau$ so $\Delta_Y(\pi^*(i), \pi^*(i')) \ge 2 \tau$. Therefore $s_{\pi^*(i)} < s_{\pi^*(i')}$. According to  the construction of $\tilde \pi$,  $\tilde \pi^{-1} \circ \pi^*(i) < \tilde \pi^{-1} \circ \pi^*(i')$.
\end{proof}

Next, recall that for a matrix $A \in \mathcal S^m$, $\mathcal J$ denotes the set of pairs of indices $(i,j) \in [n]^2$ such that $A_{i,\cdot}$ and $A_{j,\cdot}$ are not identical. The quantity $R(A)$ is defined by
\[
R(A) = \frac 1n \max_{\substack{\mathcal I \subset [n]^2\\ |\mathcal I| = n}} \sum_{(i, j) \in \mathcal I \cap \mathcal J} \Big( \frac{\|A_{i,\cdot} - A_{j,\cdot}\|_2^2}{\|A_{i,\cdot} - A_{j,\cdot}\|_\infty^2}\wedge \frac{m \|A_{i,\cdot} - A_{j,\cdot}\|_2^2}{\|A_{i,\cdot} - A_{j,\cdot}\|_1^2} \Big) \,.
\]
For any nonzero vector $u \in {\R}^m$, $\|u\|_2^2/\|u\|_\infty^2 \ge 1$ with equality achieved when $\|u\|_0 = 1$, and $\|u\|_2^2/\|u\|_1^2 \ge m^{-1}$ with equality achieved when all entries of $u$ are the same. Hence $R(A) \ge 1$.
Moreover, $\|u\|_2^2 \le \|u\|_1 \|u\|_\infty$ by H\"older's inequality, so 
$\frac{\|u\|_2^2}{ \|u\|_\infty^2} \wedge \frac{m \|u\|_2^2}{\|u\|_1^2} \le \sqrt m$
as the product of the two terms is no larger than $m$. The equality is achieved by $u = (1, \dots, 1, 0, \dots, 0)$ where the first $\sqrt m$ entries are equal to one. Therefore,
\[ R(A) \in \big[1, \sqrt m \, \big] \,.\]
Intuitively, the quantity $R(A)$ is small if the difference of any two rows of $A$ is either very sparse or very dense.

\begin{lemma} \label{lem:approx-error}
There is an event $\mathcal E$ of probability at least $1- \delta$ on which
\[ \|\tilde \Pi A^* - \Pi^* A^*\|_F^2 \lesssim \sigma^2 R(A^*) \, n  \log(nm\delta^{-1}) \,. \]
\end{lemma}

\begin{proof}
Throughout the proof, we restrict ourselves to the event $\mathcal E$ defined in Lemma~\ref{lem:row-gap}. To simplify the notation, we define $\alpha_i = A^*_{\tilde \pi^{-1} \circ \pi^*(i),\cdot} - A^*_{i, \cdot}$. Then
\begin{equation} \label{eq:sum-alpha}
\|\tilde \Pi A^* - \Pi^* A^*\|_F^2 = \sum_{i=1}^n \|A^*_{\tilde \pi(i),\cdot} - A^*_{\pi^*(i), \cdot}\|_2^2 
= \sum_{i \in I} \|\alpha_i\|_2^2 \,,
\end{equation}
where $I$ is the set of indices $i$ for which $\alpha_i$ is nonzero.
For each $i \in I$,
\begin{align}
\|\alpha_i\|_2^2 &= \min\Big( \frac{\|\alpha_i\|_2^2}{\|\alpha_i\|_\infty^2}, \frac{m\|\alpha_i\|_2^2}{\|\alpha_i\|_1^2}\Big)  \cdot \max \Big(\|\alpha_i\|_\infty^2, \frac{ \|\alpha_i\|_1^2}{m} \Big) \nonumber \\
&= \min\Big( \frac{\|\alpha_i\|_2^2} {\|\alpha_i\|_\infty^2}, \frac{m\|\alpha_i\|_2^2}{\|\alpha_i\|_1^2}\Big)  \cdot \Delta_{A^*} \big(i, \tilde \pi^{-1} \circ \pi^*(i) \big)^2 \label{eq:min-del}
\end{align}
by \eqref{eq:def-del-app}.

Next, we proceed to showing that $|\Delta_{A^*}(i, \nu(i))| \le 4\tau$ for any $i \in [n]$, where $\nu= \tilde \pi^{-1} \circ \pi^*$. 
To that end, note that if $\Delta_{A^*}(i, \nu(i)) > 4\tau$, in which case $\Delta_{A^*}(i, i') > 4\tau$ for all $i' \in I':=\{i' \in [n]\,:\, i'\ge \nu(i)\}$, then  
it follows from Lemma~\ref{lem:row-gap} that on $\cE$,  $\nu(i) < \nu(i'), \ \forall\  i \in I'$. Note that $|\nu(I')|=|I'|=n-\nu(i)+1$. Hence $\nu(i) < \nu(i'), \ \forall\  i \in I'$ implies that  $\nu(i)\le n-|\nu(I')|=\nu(i)-1$, which is a contradiction. Therefore, there does not exist such $i \in [n]$ on $\cE$. The case where $\Delta_{A^*}(i, \nu(i)) < -4\tau$ is treated in a symmetric manner.

Combining this bound with \eqref{eq:sum-alpha} and \eqref{eq:min-del}, we conclude that
\begin{align*}
\|\tilde \Pi A^* - \Pi^* A^*\|_F^2 & \lesssim \sum_{i \in I} \min\Big( \frac{\|\alpha_i\|_2^2}{\|\alpha_i\|_\infty^2}, \frac{m\|\alpha_i\|_2^2}{\|\alpha_i\|_1^2}\Big)  \cdot \tau^2 \\
&\lesssim \sigma^2 R(A^*) \, n \log(nm\delta^{-1}) \,.
\end{align*}
by the definitions of $R(A^*)$ and $\tau$.
\end{proof}

\subsection{Proof of Theorem~\ref{thm:monotone}}

The bound is an immediate consequence of Lemma~\ref{lem:approx-est} and Lemma~\ref{lem:approx-error} with $\delta=(nm)^{-C}$ for $C > 0$.



\section{Upper bounds in a trivial case} \label{sec:trivial-upper}

In Theorem~\ref{thm:lower-global}, we have observed the term $\frac{\sigma^2}{m} \wedge m^2 V(A)^2$, whereas the LS estimator only has $\frac{\sigma^2}{m} \log n$ in the upper bounds. The next proposition shows that in the case $m^2 V(A)^2 \le \frac{\sigma^2}{m}$, we can simply use an averaging estimator that achieves the term $m^2 V(A)^2$.

\begin{proposition} \label{prop:special-regime}
For $Y = \Pi^* A^* + Z$ where $Z \sim \operatorname{subG}(\sigma^2)$, let $\hat \Pi = I_n$ and $\hat A$ be defined by $\hat A_{i,j} = \frac 1n \sum_{k=1}^n Y_{k,j}$ for all $(i,j) \in [n]\times [m]$. Then,
\[ \frac{1}{nm} \| \hat \Pi\hat A - \Pi^* A^*\|_F^2\lesssim \frac{\sigma^2} {n} +  m^2 V(A)^2 \]
with probability at least $1-\exp(-m)$ and 
\[ \frac{1}{nm}\E \| \hat \Pi\hat A - \Pi^* A^*\|_F^2\lesssim \frac{\sigma^2} {n} +  m^2 V(A)^2 \, .\]
\end{proposition}

\begin{proof}
Recall that $V(A) = (\frac 1m \sum_{j=1}^m V_j(A)^{2/3} )^{3/2}$. Since the $\ell_2$-norm of a vector is no larger than the $\ell_{\frac 23}$-norm,
\[ \sum_{j=1}^m V_j(A)^2 \le \Big(\sum_{j=1}^m V_j(A)^{2/3} \Big)^{3} = m^{3} V(A)^2. \]
On the other hand,
\[ \hat A_{i,j} = \frac 1n \sum_{k=1}^n A^*_{k,j} + \frac 1n \sum_{k=1}^n Z_{k,j} \, ,\]
so we have that
\begin{align*} 
& \quad \  \|\hat \Pi \hat A - \Pi^* A^*\|_F^2 \\
&= \sum_{i \in [n], j \in [m]} \Big(\frac 1n \sum_{k=1}^n A^*_{k,j} + \frac 1n \sum_{k=1}^n Z_{k,j} - A^*_{i, j} \Big)^2 \\
&\le 2 \sum_{i \in [n], j \in [m]} \Big(\frac 1n \sum_{k=1}^n A^*_{k,j} - A^*_{i, j} \Big)^2 + \frac 2{n^2} \sum_{i \in [n], j \in [m]} \Big(\sum_{k=1}^n Z_{k,j} \Big)^2 \\
&\le 2 n \sum_{j \in [m]} V_j(A)^2 + \frac 2n \sum_{j \in [m]} \Big(\sum_{k=1}^n Z_{k,j} \Big)^2 \\
& \le 2 n m^3 V(A)^2 + 2 \sum_{j \in [m]} g_j^2 \, ,
\end{align*}
where $g_j = \frac{1}{\sqrt n} \sum_{k=1}^n Z_{k,j}$ for $j \in [m]$ so that $g_1, \dots, g_m$ are centered sub-Gaussian variables with variance proxy $\sigma^2$. 
It is well-known that ${\E} g_j^2 \lesssim \sigma^2$, so
\[ {\E} \|\hat \Pi \hat A - \Pi^* A^*\|_F^2 \lesssim n m^3 V(A)^2 + m \sigma^2. \]
Moreover, since $(g_1, \dots, g_m)$ is a sub-Gaussian vector with variance proxy $\sigma^2$, it follows from \cite[Theorem~2.1]{HsuKakZha12} that
$\sum_{j=1}^m g_j^2 \lesssim  \sigma^2 m$
with probability at least $1-\exp(-m)$. On this event,
\[ \|\hat \Pi \hat A - \Pi^* A^*\|_F^2 \lesssim  n m^3 V(A)^2 + m \sigma^2. \]
Dividing the previous two displays by $nm$ completes the proof.
\end{proof}



\section{Unimodal regression} \label{sec:unimodal}

If the permutation in the main model \eqref{eq:model} is known, then the estimation problem simply becomes a concatenation of $m$ unimodal regressions. In fact, our proofs imply new oracle inequalities for unimodal regression. Recall that $\mathcal U$ denotes the cone of unimodal vectors in ${\R}^n$.  Suppose that we observe
\[y = \theta^* + z \,,\] 
where $\theta^* \in {\R}^n$ and $z$ is a sub-Gaussian vector with variance proxy $\sigma^2$. Define the LS estimator $\hat \theta$ by 
\[ \hat \theta \in \operatorname*{argmin}_{\theta \in \mathcal U} \|\theta - y\|_2^2 \,. \] 
Moreover let $k(\theta) = \card(\{\theta_1, \dots, \theta_n\})$ and $V(\theta) = \max_{i\in [n]} \theta_i - \min_{i \in [n]} \theta_i$.

\begin{corollary} \label{cor:unimodal}
There exists a constant $c>0$ such that with probability at least $1- n^{-\alpha}$, $\alpha \ge 1$,
\begin{equation}
\label{EQ:oiunimodal}
\frac 1n \|\hat \theta - \theta^*\|_2^2 \lesssim \min_{\theta \in \mathcal U} \Big( \frac 1n \|\theta - \theta^*\|_2^2 + \sigma^2 \frac{k(\theta)}{n} \log \frac{en}{k(\theta)} \Big) + \alpha \sigma^2 \frac{\log n}{n}
\end{equation}
and
\[ \frac 1n \|\hat \theta - \theta^*\|_2^2 \lesssim  \min_{\theta \in \mathcal U} \Big[ \frac 1n \|\theta - \theta^*\|_2^2 + \Big(\frac{\sigma^2 V(\theta) \log n}{n}\Big)^{2/3} \Big]  + \alpha \sigma^2 \frac{\log n}{n} \,.
\]
The corresponding bounds in expectation also hold.
\end{corollary}

\begin{proof}
The proof closely follows that of Theorem~\ref{thm:adaptive} and Theorem~\ref{thm:global}.


First note that the term $n \log n$ in the bound of Lemma~\ref{lem:cover-matrix-union} comes from a union bound applied to the set of permutations, so it is not present if we consider only the set of unimodal matrices $\mathcal U^m$ instead of $\mathcal M$. Hence taking $m=1$ in the lemma yields that
\[ \log N \big(\Theta_{\mathcal U} (\tilde \theta, t), \|\cdot\|_2, \varepsilon \big) \le C \varepsilon^{-1} t\, k(\tilde \theta) \log \frac{en}{k(\tilde \theta)} \, . \]

For $\tilde \theta \in \mathcal U$, define
\[ f_{\tilde \theta}(t) = \sup_{\theta \in \mathcal U \cap \mathcal B^{n}(\tilde \theta, t)} \langle \theta - \tilde \theta, y - \tilde \theta \rangle - \frac {t^2}2 \, . \]
Following the proof of Lemma~\ref{lem:f-bound} and using the above metric entropy bound, we see that
\[ f_{\tilde \theta}(t) \le C \sigma t \sqrt{k(\tilde \theta) \log \frac{en}{k(\tilde \theta)}} + t \|\tilde \theta - \theta^*\|_2 - \frac{t^2}2 + st \]
with probability at least $1-C \exp(-\frac{cs^2}{\sigma^2 })$.
Then the proof of Theorem~\ref{thm:adaptive} gives that with probability at least $1-C \exp(-\frac{cs^2}{\sigma^2 })$,
\[
\|\hat \theta - \theta^*\|_2 \le C \Big(\sigma \sqrt{k(\tilde \theta) \log \frac{en}{k(\tilde \theta)}} + \|\tilde \theta - \theta^*\|_2 \Big) + 2s \,.
\]
Taking $s=C \sigma \sqrt{\alpha \log n}$ for $\alpha \ge 1$ and $C$ sufficiently large, we get that with probability at least $1- n^{-\alpha}$,
\[
\|\hat \theta - \theta^*\|_2^2 \lesssim \sigma^2 k(\tilde \theta) \log \frac{en}{k(\tilde \theta)} +  \|\tilde \theta - \theta^*\|_2^2 + \alpha \sigma^2 \log n \,.\]
Minimizing over $\tilde \theta \in \mathcal U$ yields \eqref{EQ:oiunimodal}.
The corresponding bound in expectation follows from integrating the tail probability as in the proof of Theorem~\ref{thm:adaptive}. 

Finally, we can apply the proof of Theorem~\ref{thm:global} with $m=1$ to achieve the global bound.
\end{proof}

Note that the bounds in Corollary~\ref{cor:unimodal} match the minimax lower bounds for isotonic regression in \cite{BelTsy15} up to logarithmic factors. Since every monotonic vector is unimodal, lower bounds for isotonic regression automatically hold for unimodal regression. Therefore, we have proved that the LS estimator is minimax optimal up to logarithmic factors for unimodal regression.

A result similar to~\eqref{EQ:oiunimodal} was obtained by Bellec in the revision of~\cite{Bel15} that was prepared independently and contemporaneously to this paper. Chatterjee and Lafferty also improved their bounds to having optimal exponents \cite{ChaLaf15} after the first version of our current paper was posted. Interestingly Bellec employs bounds on the statistical dimension by leveraging results from~\cite{Ameetal14}, and Chatterjee and Lafferty use both the variational formula and the statistical dimension. Moreover, their results are presented in the well-specified case where $\theta^* \in \cU$ and $\theta =\theta^*$.

\newcommand{\etalchar}[1]{$^{#1}$}

\end{document}